\documentclass[12pt]{amsart}

\usepackage{amsmath}
\usepackage{amsfonts}
\usepackage{amssymb}
\usepackage{mathabx} 
\usepackage{color} 
\usepackage{graphicx}
\usepackage{subcaption}
\newcommand{\begeq} {\begin{equation}}
\newcommand{\eneq} {\end{equation}}
\newcommand{\begpf} {\begin {proof}}
\newcommand{\enpf} {\end {proof}}

\def\dst{\displaystyle}

\newcommand{\abs}[1]{\ensuremath{\left| #1 \right| }}
\newcommand{\norm}[1]{\lVert#1\rVert}
\def\sinc{\operatorname{sinc}} 
\newcommand{\ZZ}{\mathbb{Z}}
\newcommand{\CC}{\mathbb{C}}
\newcommand{\cB}{\mathcal{B}}
\newcommand{\inv}{^{-1}}
\def\supp{\operatorname{supp}}
\def\vect{\operatorname{span}}

\def\dst{\displaystyle}

\newtheorem{lemma}{Lemma}[section]
\newtheorem{prop}[lemma]{Proposition}
\newtheorem{theorem}[lemma]{Theorem}
\newtheorem{coro}[lemma]{Corollary}

\theoremstyle{definition}
\newtheorem{definition}[lemma]{Definition}
\newtheorem{example}[lemma]{Example}

\theoremstyle{remark}

\newtheorem{remark}[lemma]{Remark}

\newtheorem{conjecture}{Conjecture}
\newtheorem{problem}[conjecture]{Problem}

\def\C{{\mathbb{C}}}

\def\N{{\mathbb{N}}}

\def\R{{\mathbb{R}}}

\def\Z{{\mathbb{Z}}}

\def\ss{{\mathcal S}}

\def\un{\mathbf{1}}

\newcommand{\scal}[1]{{\left\langle{#1}\right\rangle}}

\newcommand{\newmphi}{\kappa_\phi}
\newcommand{\newdelta}{\Delta_m}

\title{Sampling the flow of a bandlimited function}

\author[A. Aldroubi ]{Akram Aldroubi }
\address{Department of Mathematics, Vanderbilt University, Nashville, TN, 37240, USA}
\email{akram.aldroubi@vanderbilt.edu}

\author[K. Gr\"ochenig]{Karlheinz Gr\"ochenig}
\address{Faculty of Mathematics, University of Vienna, Oskar-Morgenstern-Platz 1, 1090 Vienna, Austria}
\email{karlheinz.groechenig@univie.ac.at}

\author[L. Huang]{Longxiu Huang }
\address{Department of Mathematics, University of California, Los Angeles, CA, 90095, USA}
\email{ huangl3@math.ucla.edu}

\author[Ph. Jaming]{Philippe Jaming}
\address{Univ. Bordeaux, IMB, UMR 5251, F-33400 Talence, France. CNRS, IMB, UMR 5251, F-33400 Talence, France}
\email{Philippe.Jaming@math.u-bordeaux.fr}

\author[I. Krishtal ]{Ilya Krishtal}
\address{Department of Mathematical Sciences, Northern Illinois University, DeKalb, IL 60115, USA}
\email{ikrishtal@niu.edu}

\author[J.-L. Romero]{Jos\'e Luis Romero}
\address{Faculty of Mathematics, University of Vienna, Oskar-Morgenstern-Platz 1, 1090 Vienna, Austria\\and\\Acoustics Research Institute, Austrian Academy of Sciences, Wohl\-leben\-gasse 12-14, 1040 Vienna, Austria}
\email{jose.luis.romero@univie.ac.at, jlromero@kfs.oeaw.ac.at}

\begin{document}

\begin{abstract}
We analyze the problem of reconstruction of a bandlimited function $f$
from the space-time samples of its states $f_t=\phi_t\ast f$
resulting from the convolution with a kernel $\phi_t$. It is
well-known that, in natural phenomena, uniform space-time samples of
$f$  are not sufficient to reconstruct $f$ in a stable way. To enable
stable reconstruction, a space-time sampling with periodic nonuniformly spaced samples must be used as was shown by Lu and Vetterli. We show that the stability of reconstruction, as measured by a condition number, controls the maximal gap between the spacial samples. We provide a quantitative statement of this result. In addition, instead of irregular space-time samples, we show that uniform dynamical samples at sub-Nyquist spatial rate allow one to stably reconstruct the function $\widehat f$ away from certain, explicitly described blind spots. We also consider several classes of finite dimensional subsets  of bandlimited functions in which the stable reconstruction is possible, even inside the blind spots. We obtain quantitative estimates for it using Remez-Tur\'an type inequalities. En route, we obtain a Remez-Tur\'an inequality for prolate spheroidal wave functions. 
To illustrate our results, we present some numerics and explicit estimates for the heat flow problem. 
 
\end{abstract}

\maketitle

\section{Introduction}

In this paper, we consider the sampling and reconstruction problem of signals $u = u(t,x)$ that arise as an evolution of an initial signal $f = f(x)$ under the action of convolution operators. The initial signal $f$ is assumed to be
in the Paley-Wiener space $PW_c$, $c>0$ (fixed throughout this paper) given by 
\begin{align*}
PW_{c} := \left\{f \in L^2(\mathbb{R}):\supp(\widehat{f}) \subseteq [-c,c]\right\}
\end{align*}
with the Fourier transform  normalized as $\widehat{f}(\xi)=\int_\R f(t)e^{- it\xi}\,\mathrm{d}t$.

The functions $u$ are solutions of  initial value problems stemming from a physical  system. Thus, due to the semigroup properties of such solutions, there is a family of kernels $\{\phi_t: t> 0\}$ such that  $u(t,x)=\phi_t\ast f(x)$, $\phi_{t+s}=\phi_t\ast\phi_s$ for all $t,s\in (0,\infty)$,   and $f= \lim\limits_{t\to 0+}\phi_t\ast f$,  $f\in L^2(\R)$.

As we are primarily interested in physical systems, we typically consider the following set of kernels: 
\begin{align}
\label{DefOp}
\Phi_c= \{ \phi\in L^1(\mathbb{R}): \mbox{there exists }\newmphi>0\;\mbox{such that }  \newmphi\le\widehat \phi(\xi) \le 1 \text{ for } |\xi|\leq c, \widehat \phi(0)=1 \}.
\end{align}

Observe that $\phi\in L^1$ implies that $\widehat\phi$ is continuous and, therefore, the existence of $\newmphi > 0$ such that $\widehat \phi\geq \newmphi$ on $[-c,c]$ is equivalent to $\widehat \phi>0$ on $[-c,c]$. We remark that some of our results hold for a less restrictive class of kernels.

\begin{example}\label{ex1}
A prototypical example is the diffusion process with $\widehat\phi_t(\xi) = e^{-t\sigma^2\xi^2}$, $t>0$
It corresponds to the initial value problem (IVP) for the heat equation (with a diffusion parameter $\sigma\not=0$)
\begin{equation}\label{heateq}
\begin{cases}\partial_t u(x,t)=\sigma^2\partial_x^2u(x,t)&\mbox{for }x\in\mathbb{R}\mbox{ and }t>0\\
u(x,0)=f(x)&\end{cases},\end{equation} for which the solution is given by 
$u(x,t)=(\phi_t*f)(x)$.

Other examples include the IVP for 
the fractional diffusion equation $$\begin{cases}\partial_t u(x,t)=(\partial_x^2)^{\alpha/2} u(x,t)&\mbox{for }x\in\R\mbox{ and }t>0\\
u(x,0)=f(x)&\end{cases}, 0<\alpha\leq 1,$$ for which the solution is given by 
$u(x,t)=(\phi_t*f)(x)$ with $\widehat\phi_t(\xi) = e^{-t|\xi|^\alpha}$,
and the IVP for
the Laplace equation in the upper half plane $$\begin{cases}\Delta u(x,y)=0&\mbox{for }x\in\R\mbox{ and }y>0\\
u(x,0)=f(x)&\end{cases},$$ for which the solution is given by 
$u(x,y)=(\phi_y*f)(x)$ with %
$\widehat\phi_y(\xi) := e^{-y|\xi|}$.
\end{example}

The following problem serves as a motivation for  this paper.%

\begin{problem}\label{pro1}
Let $ \phi \in \Phi$, $L>0$, and $\Lambda\subset \R$ be a discrete subset of $\R$. What are the conditions that
allow one to recover a function $f \in PW_c$  in a stable way from the data set  
\begin{equation}
\label {meas}	
\{(f*\phi_t)(\lambda): \lambda\in \Lambda, 0 \leq t \leq L\}?
\end{equation}	 
The set of measurements \eqref {meas} is the image of an operator 
$\mathcal{T}:PW_c\to L^2\big(\Lambda\times [0,L]\big)$. 
Thus, the stable recovery of $f$ from \eqref {meas} amounts to finding conditions on $\Lambda, \phi$ and $L$ such 
that $\mathcal T$ has a bounded inverse from $\mathcal T(PW_c)$ to $PW_c$ or,	
equivalently, %
the existence of  $A,B>0$ such that 
	\begin{equation}\label{SemiContFrL}
	A \norm{f}_2^2 \leq \int_0^L \sum_{\lambda \in \Lambda} \abs{(f*\phi_t)(\lambda)}^2 dt \leq B \norm{f}_2^2, \text{ for all } f \in PW_c. 
	\end{equation}
	If for a given $\phi$ and $L$ the frame condition \eqref {SemiContFrL} is satisfied, we say that $\Lambda=\Lambda_{\phi,L}$ is a stable sampling set. %
\end{problem}

 \begin{remark}
  It was shown in \cite[Theorem 5.5]{AHP19}  that  
  $\Lambda_{\phi,L}$ is a stable sampling set 
  for some $L>0$, if and only if 
 $\Lambda_{\phi,1}$ is a stable sampling set.
 Thus, for qualitative results, 
  we will only  consider the case of $L=1$. For quantitative results, however, we may keep $L$ in order to estimate the optimal time length of measurements.
\end{remark}

 \begin{remark}
Whenever \eqref{SemiContFrL} holds, standard frame methods can be used for the stable reconstruction of $f$ \cite{FR05}.
\end{remark}

Let us discuss Problem  \ref{pro1} in more detail in the case of our prototypical example.

\subsection{Sampling the heat flow}
Consider the problem of sampling the temperature in a heat diffusion process initiated by a bandlimited 
function $f \in PW_c$:
\begin{align*}
f_t := f * \phi_t, \qquad 0 \leq t \leq 1,
\end{align*}
where $\phi_t$ is the heat kernel at time $t$:
\begin{align}\label{hk}
\widehat{\phi_t}(\xi) = e^{-t\sigma^2\xi^2},
\end{align}
with a parameter $\sigma\not=0$. 
According to Shannon's sampling theorem, $f$ can be stably reconstructed from
equispaced samples $\{f(k/T): k \in \mathbb{Z}\}$ if and only if the sampling rate $T$ is bigger than or 
equal to the critical value $T =\displaystyle \frac{c}{\pi}$, known as the \emph{Nyquist rate}. The Nyquist bound 
is universal in the sense that it also applies to irregular sampling patterns: if a bandlimited function can be 
stably reconstructed from its samples at $\Lambda \subseteq \mathbb{R}$, then the \emph{lower Beurling density}
\[ D^-(\Lambda):=\liminf_{r\rightarrow\infty}\inf_{x\in\mathbb{R}}\frac{\#(\Lambda\bigcap [x-r,x+r])}{2r}  \]
satisfies $D^{-}(\Lambda) \geq \displaystyle \frac{c}{\pi}$. Recall that the \emph{upper Beurling density} 
is defined by 
\[ D^+(\Lambda):=\limsup_{r\rightarrow\infty}\sup_{x\in\mathbb{R}}\frac{\#(\Lambda\bigcap [x-r,x+r])}{2r}.  \]

We are interested to know  if the spatial sampling rate can be reduced by using the information provided by the following spatio-temporal samples:
\begin{equation}
\label{eq_st_samp1}
\{f_t(k/T): k \in \mathbb{Z}, 0 \leq t \leq 1\}.
\end{equation}
Observe that the amount of collected data in \eqref{eq_st_samp1} is not smaller than that in the case of sampling 
at the Nyquist rate $T=\displaystyle \frac{c}{\pi}$. If $T <\displaystyle \frac{c}{\pi}$, however,  the density of 
sensors is smaller, and thus such a sampling procedure may provide considerable cost savings.

Lu and Vetterli showed \cite{Lu_2009} that for all $T < \displaystyle \frac{c}{\pi}$ there exist
bandlimited signals with norm 1 that almost vanish on the samples \eqref{eq_st_samp1}, i.e.~stable reconstruction 
is impossible from \eqref{eq_st_samp1}. As a remedy, they introduced periodic, nonuniform
sampling patterns $\Lambda \subseteq \mathbb{R}$ that do lead to a meaningful \emph{spatio-temporal trade-off}: 
there exist sets $\Lambda \subseteq \mathbb{R}$ that have sub-Nyquist density and, yet, lead to the frame inequality:
\begin{equation}\label{SemiContFr}
	A \norm{f}_2^2 \leq \int_0^1 \sum_{\lambda \in \Lambda} \abs{(f_t)(\lambda)}^2 \mathrm{d}t \leq B \norm{f}_2^2, 
	\text{ for all } f \in PW_c,
	\end{equation}
with $A,B>0$; see Example \ref{ex_low_dens} for a concrete construction. The emerging field of dynamical sampling 
investigates such phenomena in great generality (see, e.g., \cite{ACCMP17, ACMT17, ADK13, ADK15, AHP19}). 

As follows from Example \ref{ex_low_dens}, the estimates \eqref{SemiContFr} may hold with an arbitrary small sensor density. The meaningful trade-off between spatial and temporal resolution, however, is limited by the desired numerical accuracy. For example, in the following theorem we relate the \emph{maximal gap} of a stable sampling set  to the bounds from  \eqref{SemiContFr}.

\begin{theorem}\label{thm_gap_intro}
Let $\Lambda\subseteq\mathbb{R}$ be such that \eqref{SemiContFr} holds. Then there exists an absolute constant $K>0$ such that, for $\dst R\geq K\max\left(\frac{B}{A},\frac{1}{c}\right)$ and every $a\in\mathbb{R}$, we have 
$[a-R,a+R]\cap\Lambda\neq\emptyset$.
In particular, we have $D^-(\Lambda)\geq K^{-1}\min\left(\frac{A}{B},c\right)$ and $D^+(\Lambda)\leq KB$.
\end{theorem}

Theorem  \ref{ConGapThm}, which is a more general version of the above result, provides a more explicit dependence of $K$ on the   parameters of the problem.

Besides the constraints implied by Theorem \ref{thm_gap_intro},
the special sampling configurations of Lu and Vetterli that lead to \eqref{SemiContFr} lack the simplicity of regular sampling patterns.
In this article, we explore a different solution to the diffusion sampling problem. We consider sub-Nyquist equispaced spatial sampling patterns \eqref{eq_st_samp1} with $T=\displaystyle \frac{c}{m\pi}$,
$m \in \mathbb{N}$, and restrict the sampling/reconstruction problem to a subset $V \subseteq PW_c$, aiming for an inequality of the form:
\begin{align}
\label{eq_diff_samp}
A\|f\|_2^2 \le \int_0^1 \sum_{k\in\ZZ}\left|f_t\left(\frac{m\pi}{c}k\right)\right|^2\,\mathrm{d}t \le B\|f\|_2^2,
\qquad f \in V.
\end{align}
Specifically, we consider the following signal models.
\medskip

\noindent {\bf Away from blind spots}. We will identify a set $E$ with measure arbitrarily close to $1$ such that \eqref {eq_diff_samp} holds with $V=V_E=\{f\in PW_c:  \supp  \widehat f \subseteq E\}$. In effect, $E$ is the set 
$[-c,c]\setminus\mathcal{O}$ where $\mathcal{O}$ is a small open neighborhood of a  finite set, i.e., $E$ avoids a certain number of ``blind spots.'' 

 \begin{theorem}\label{sufconT0}
Let $ \phi \in \Phi$ and $m\geq 2$ be an integer. Then for any  $r>0$ there exists a certain compact set $E\subseteq [-c,c]$ of measure at least $2c-r$ such that %
 any $f\in V_E$ can be recovered from the samples 
$$
\mathcal M =\left\{f_t\left(\frac{m \pi}{c}k\right): k \in \mathbb{Z}, 0 \leq t \leq 1\right\}
$$
in a stable way. %
\end{theorem}

The set $E$ in the above theorem depends only on $\phi$ and the choice of $r$. The stable recovery in this case means that \eqref {eq_diff_samp} holds with $B=1$ and some $A > 0$ which is estimated in
a more explicit version of the above result,  
Theorem \ref {sufconT}.

\medskip

\noindent {\bf Prolate spheroidal wave functions}.
The Prolate Spheroidal Wave Functions (PSWFs) are eigenfunctions of an integral operator known as the time-band liminting operator or sinc-kernel operator
$$
\mathcal{Q}_cf(x)=\int_{-1}^1\frac{\sin \pi c(y-x)}{\pi(y-x)}f(y)\,\mbox{d}y.
$$
Using the min-max theorem, we get that $\psi_{n,c}$ is the norm-one solution of
the following extremal problem
$$
\max\left\{\frac{\norm{f}_{L^2(-1,1)}}{\norm{f}_{L^2(\R)}}\,: f\in PW_c,\ 
f\in\vect\{\psi_{k,c}:\ k<n\}^\perp\right\}
$$
where the condition $f\in\vect\{\psi_{k,c}:\ k<n\}^\perp$ is void for $n=0$.
The family $(\psi_{n,c})_{n\geq 0}$ forms an orthogonal basis for $PW_c$ and has the property to form
an orthonormal sequence in $L^2(-1,1)$.

We consider the $N$-dimensional space
\begin{equation}
\label{eq_moda}
V_N=\mbox{span}\{\psi^c_0,\ldots,\psi^c_N\}\subset PW_c.
\end{equation}
The Landau-Pollak-Slepian theory shows that this subspace provides an optimal approximation
of a bandlimited function that is concentrated on $[-1,1]$. More precisely, $V=V_N$ minimizes the approximation error
\begin{align*}
\sup_{\stackrel{f \in PW_c}{\norm{f}_2=1}} \inf_{g \in V} \int_{-1}^{1} \abs{f(x)-g(x)}^2\, \mbox{d}x,
\end{align*}
among all $N$-dimensional subspaces of $PW_c$.

\medskip

\noindent {\bf Sparse sinc translates with free nodes}.
In this model, we let
\begin{equation}
\label{eq_modb}
V_N =\left\{\sum_{n=1}^N c_n\sinc c(x-\lambda_n)\,: c_1,\ldots,c_N\in\C,\ \lambda_1,\ldots,\lambda_N\in\R\right\}
\end{equation}
be the class of linear combinations of $N$ arbitrary translates of the sinc kernel 
$ \sinc(x)=\frac{\sin x}{ x} $
Note that $V_N$ is not a linear space.
However,  $V_N - V_N \subseteq V_{2N}$. Therefore, \eqref{eq_diff_samp} with $V=V_{2N}$ implies
\begin{align*}
A\|f-g\|_2^2 \le \int_0^1 \sum_{k\in\ZZ}\left|f_t\left(\frac{m\pi}{c}k\right)-g_t\left(\frac{m\pi}{c}k\right)\right|^2\,\mathrm{d}t \le B\|f-g\|_2^2,
\qquad f,g \in V_N,
\end{align*}
which ensures the numerical stability of the  sampling problem $f \mapsto \{f_t(m\pi k/c): k \in \mathbb{Z}: 0 \leq t \leq 1 \}$ restricted non-linearly to the class $V_N$. In other words, if \eqref{eq_diff_samp} holds with $V=V_{2N}$ then any $f\in V_N$ can be stably reconstructed from the samples \eqref{eq_st_samp1}.

\medskip

\noindent {\bf Fourier polynomials}.
As our last model, we consider the Fourier image of the space of polynomials of degree at most $N$ restricted  to the unit interval. Explicitly,
\begin{equation}
\label{eq_modc}
V_N =\left\{\sum_{n=0}^N c_n D^n\sinc c\cdot: c_0,\ldots,c_N\in\C\right\},
\end{equation}
where $D: PW_c \to PW_c$ is the differential operator  $Df = f^\prime$.
Observe that the union of such $V_N$, $N\in\N$, is dense in $PW_c$.

\medskip

In this article, we show that each of the above-mentioned signal models regularizes the diffusion sampling problem, albeit with possibly very large condition numbers.
\begin{theorem}
\label{thm:gauss}
Let $m \geq 2$ be an integer, $\Phi$ be given by $\widehat{\Phi}(\xi)=e^{-\sigma^2\xi^2}$. 
Let $V=V_N$ be given by \eqref{eq_moda}, 
\eqref{eq_modb}, or \eqref{eq_modc}. 
Then \eqref{eq_diff_samp} holds with
\begin{align*}
A =\frac{c\kappa_0(c)}{(\sigma c)^2+m}
\exp\Bigl(-\kappa_1(c)N-m^2\bigl(-\kappa_2(c)\ln\sigma+\kappa_3(c)\sigma^2+\ln m\bigr)\Bigr)
, \qquad 
B=1,
\end{align*}
where the $\kappa_j$'s are positive constants that depend on $c$ only. 
\end{theorem}

We provide a more precise expression for the lower frame constant in Theorem \ref{thm:gauss_body}. %
Note that the lower bound deteriorates when $\sigma^2 \to 0$ (no diffusion) and
$\sigma^2 \to +\infty$ (very rapid diffusion). This agrees with the intuition and numerical experiments for 
(non-bandlimited) sparse initial conditions presented in \cite{Ranieri2011SamplingAR}: if $\sigma^2$ is very small, 
because of spatial undersampling, some components of $f$ may be hidden from the sensors, while for large 
$\sigma^2$ the diffusion completely blurs out the signal and no information can be extracted.

\begin{remark} To simplify the discussion we take $c=1/2$ in this remark.
There are instances when  Theorem \ref{thm:gauss} applies for a signal $f\in V_N$ which cannot be recovered simply from its samples on, say, $2\Z$. As an example, we offer $V_1$ given by \eqref{eq_modb} with $\lambda_1 = 1$. The samples at time $t=0$ are not sufficient to identify each signal since
$\sinc(\cdot-1)\in V_N$  vanishes on $m \mathbb{Z}$, $m\ge 2$.  Similarly, for Theorem \ref  {sufconT0}: the function  $\sin (\omega \cdot) \sinc (\frac {\cdot} a)$, with an appropriately chosen $a$ and $\omega$, belongs to $V_E$ and vanishes on $m\Z$ for $m\ge 2$. In finite dimensional subspaces $V_N$, e.g., given by  \eqref {eq_moda} and \eqref {eq_modc},  sampling at time $t=0$ with any $m \in \N$ may be sufficient for stable recovery. However, the expected error of reconstruction in the presence of noise will be reduced if temporal samples are used in addition to those at $t=0$. 
Theorems \ref {sufconT0} and \ref {thm:gauss}   can be used together. For example, a function $f$ can be reconstructed away from the blind spots using Theorem \ref {sufconT0} and approximated around the blind spots using Theorem \ref {thm:gauss}.
\end{remark}

\subsection{Technical overview}\label{tech}

Lu and Vetterli explain the impossibility of subsampling the heat-flow of a bandlimited function on a grid \eqref{eq_st_samp1} as follows \cite{Lu_2009}. The function with Fourier transform
\begin{align*}
\widehat f := \delta_{-T} - \delta_{T}
\end{align*}
is formally bandlimited to $I =[-c,c]$ if $T<c$, and vanishes on the lattice
$\tfrac{\pi}{T} \mathbb{Z}$. Moreover, $f$ is an eigenfunction
of the diffusion operator since
\begin{align*}
\widehat f_t = e^{-t\sigma^2(-T)^2} \delta_{-T} - e^{-t\sigma^2T^2} \delta_{T} = 
e^{-t\sigma^2T^2} \widehat f,
\end{align*}
see \eqref{heateq} and \eqref{hk}.
Hence, all the diffusion samples \eqref{eq_st_samp1} vanish, although $f \not\equiv 0$.
While no Paley-Wiener function is infinitely concentrated at $\{-T,T\}$, a more formal argument can be given by regularization. If $\eta: \mathbb{R} \to \mathbb{R}$
is continuous and supported on $[-1,1]$ and $\eta_\varepsilon(x) = \varepsilon^{-1} \eta(x/\varepsilon)$, then
$f \cdot \widehat{\eta}_\varepsilon \in PW_c$ and provides a counterexample to \eqref{SemiContFrL},
provided that $\varepsilon$ is sufficiently small.

As we show below in Subsection \ref{ssds}, a similar phenomenon holds for more general diffusion kernels $\phi$ as in \eqref{DefOp}. Indeed, an analysis along the lines of the Papoulis sampling theorem \cite{P77} shows that the diffusion samples \eqref{eq_st_samp1} of a function $f \in PW_c$ do not lead to a stable recovery of $\widehat{f}$. However, these samples do allow for the stable recovery \emph{away from certain blind spots} determined by $\phi$; that is, one can effectively recover $\widehat{f} \cdot \un_{E}$, for a certain subset $E \subseteq I$
of positive measure ($\un_E$ denotes the characteristic function of the set $E$). If we, furthermore,
 restrict the sampling problem to one of the finite dimensional spaces $V=V_N$ given by \eqref{eq_moda}, \eqref{eq_modb},
or \eqref{eq_modc}, we may then  infer all other values of $\widehat f$. The main tools, in this case, are \emph{Remez-Tur\'an-like} inequalities of the form:
\begin{align*}
\norm{\widehat{f}\un_I} \leq C_E \norm{\widehat{f}\un_E}, \qquad f \in V.
\end{align*}
For Fourier polynomials
\eqref{eq_modc} the classical Remez-Tur\'an inequality provides an explicit constant $C_E$,
while the case of sparse sinc translates \eqref{eq_modb} is due to Nazarov \cite{Na}. The 
corresponding inequality for prolate spheroidal wave functions
\eqref{eq_moda} is new and a contribution of this article (our technique relies on \cite{JKS}).

\subsection{Paper organization and contributions}

In Section \ref{ssd}, we show that uniform dynamical samples at sub-Nyquist rate allow one to stably reconstruct the function $\widehat f$ away from certain, explicitly described blind spots determined by the kernel $\phi$. We also provide an upper and lower estimate for the lower frame bound in \eqref{eq_diff_samp}. The upper estimate relies on the standard formulas for Pick matrices (see, e.g. \cite{BT17, FO01}). The lower estimate is far more intricate and is based on the analysis of certain Vandermonde matrices. We also provide some numerics and explicit estimates in the case of the %
heat flow problem.

In Section \ref{rtpbs}, we restrict the problem to the sets  $V=V_N$  given by \eqref{eq_moda}, \eqref{eq_modb}, or \eqref{eq_modc}.
We provide quantitative estimates for the frame bounds in \eqref{eq_diff_samp}. En route, we obtain an explicit Remez-Tur\'an inequality for prolate spheroidal wave functions -- a result which we find interesting in its own right.

In Section \ref{gap}, we discuss the case of irregular spacial sampling. We recall that a stable reconstruction may be possible with sets $\Lambda$ that have an arbitrarily small (but positive) lower density. Nevertheless, we show that the maximal gap between the spacial samples (and, hence, the lower Beurling density) is controlled by the condition number of the problem (i.e.~the ratio $\frac BA$ of the frame bounds).

\section{Recovering a bandlimited function away from the blind-spot}\label{ssd}

\subsection{Dynamical sampling in $PW_c$}\label{ssds}
In this section, we recall some of the results on dynamical sampling from \cite{ADK15, AHP19} and adapt them for problems studied in this paper.

For $\phi \in L^1$, consider the function
\begin{align*}
\widehat\phi_p(x)=\sum_{k\in\Z}\widehat\phi(x-2ck)\un_{[-c,c)}(x-2ck),
\end{align*}
that is, the $2c$-periodization of the piece 
of $\widehat\phi$ supported in $[-c,c)$. 
 Recall that we consider kernels from the set $\Phi$ given by \eqref{DefOp}. Hence,
\begin{align*}
\newmphi\leq\widehat\phi_p(\xi)\leq 1, \qquad \xi \in \mathbb{R}.
\end{align*}
We also write
\begin{align*}
\widehat{f_t}(\xi) := \widehat{f}(\xi) \widehat{\phi}^t(\xi), \ f\in PW_c.
\end{align*}

Next, we  introduce the \emph{sampled diffusion matrix},  %
which is  the $m\times m$ matrix-valued function given by
\begin{equation}\label{matrB}
\mathcal B_m(\xi)=\left(\int\limits^1_0\overline{(\widehat{\phi})_p^t\left(\frac{2c}{m}(\xi+j)\right)}(\widehat{\phi})_p^t\left(\frac{2c}{m}(\xi+k)\right)\,\mathrm{d}t\right)_{0\leq j,k\leq m-1}
=\int_0^1\mathcal A_m^* (\xi,t)\mathcal A_m(\xi,t)\,\mathrm{d}t, 
\end{equation}
where 
\begin{eqnarray*}
\mathcal A_{m}(\xi,t)&=&\begin{pmatrix}\dst(\widehat{\phi})_p^t\left(\frac{2c}{m}(\xi+k)\right)
\end{pmatrix}_{k=0,\ldots,m-1}\\
&=&\begin{pmatrix}\dst
(\widehat{\phi})_p^t\left(\frac{2c}{m}\xi\right)&\cdots&\dst(\widehat{\phi})_p^t\left(\frac{2c}{m}(\xi+m-1)\right)\end{pmatrix}\in\mathcal{M}_{1,m}(\mathbb{C}).
\end{eqnarray*}

\begin{remark}
\label{rem_period}
Observe that the matrix function $\mathcal B_m$ is $m$-periodic. Its eigenvalues, however, are $1$-periodic  because the   matrices $\mathcal B_m(\xi)$ and $\mathcal B_m(\xi+k)$, $k\in\Z$, are similar via a circular shift matrix.
\end{remark}

The following lemma explains the role of the sampled diffusion matrix. In the lemma, we let
\begin{equation}
\label {vectorization}
\mathbf f(\xi)=\begin{pmatrix}\dst(\widehat f)_p\left(\frac {2c}{m}(\xi+j)\right)\end{pmatrix}_{j=0,\ldots,m-1}
=\begin{pmatrix}\dst(\widehat f)_p\left(\frac {2c}{m}\xi\right)\\\vdots\\
\dst(\widehat f)_p\left(\frac {2c}{m}(\xi+m-1)\right)
\end{pmatrix}
\in\mathcal{M}_{m,1}(\mathbb{C}).
\end{equation}
Note that if we recover $\mathbf f(\xi)$ for $\xi \in [0,1]$ then we can recover  $f_p$.
Observe also that

\begin{equation}
\label{eq_note}
\begin{split}
\int_0^1\|\mathbf f(\xi)\|^2\,\mbox{d}\xi
&=\sum_{j=0}^{m-1}\int_0^1\abs{(\widehat f)_p\left(\frac{2c}{m}(\xi+j)\right)}^2\,\mbox{d}\xi 
=\frac{m}{2c}\sum_{j=0}^{m-1}\int_{2cj/m}^{2c(j+1)/m}|(\widehat f)_p(u)|^2\,\mbox{d}u\\
&=\frac{m}{2c}\int_{0}^{2c}|(\widehat f)_p(s)|^2\,\mbox{d}s
=\frac{m}{2c}\int_{-c}^{c}|\widehat f(s)|^2\,\mbox{d}s
\end{split}
\end{equation}
In other words, $f\mapsto \sqrt{\frac{2c}{m}}\mathbf f: PW_c\to L^2([0,1],\mathcal{M}_{m,1}(\mathbb{C}))$
is an isometric isomorphism.

\begin{lemma}
\label{lemma_calc}
For $f \in PW_c$,
\begin{equation}\label{eq:lister}
\int_0^1 \sum_{k\in\Z}\left|f_t\left(\frac{m \pi}{c}k\right)\right|^2\,\mathrm{d}t
=\left(\frac{c}{m\pi}\right)^2\int_0^1\mathbf f(\xi)^*\mathcal B_m(\xi)\mathbf f(\xi)\,\mathrm{d}\xi.
\end{equation}
\end{lemma}

\begin{proof}
Observe that it suffices to prove the result in $PW_c\cap \ss(\R)$ (the Schwarz class). Consider the function
\[
b(\xi,t)=\sum_{k\in\Z}f_t\left(\frac{m \pi}{c}k\right)e^{-2i\pi k\xi}.
\]
Using the Poisson summation formula and the definition of $f_t$, we get
\begin{eqnarray*}
b(\xi,t)&=&\frac{c}{m\pi}\sum_{j\in\Z}\widehat{f_t}\left(\frac{2c}{m}(\xi+j)\right)
=\frac{c}{m\pi}\sum_{-\frac{m}{2}-\xi\leq j< \frac{m}{2}-\xi}
\widehat\phi^t\left(\frac{2c}{m}(\xi+j)\right)\widehat f\left(\frac{2c}{m}(\xi+j)\right)\\
&=&\frac{c}{m\pi}\sum_{j=0}^{m-1}
(\widehat\phi)_p^t\left(\frac{2c}{m}(\xi+j)\right)(\widehat f)_p\left(\frac{2c}{m}(\xi+j)\right),
\end{eqnarray*}
Note that the  functions $b(\cdot, t)$ are $1$-periodic,
\begin{equation}\label{convert}
b(\xi,t)=\frac{c}{m\pi}\mathcal A_m(\xi,t)\mathbf f(\xi), 
\end{equation}
and thus
\[
\int_0^1|b(\xi,t)|^2\,\mathrm{d}t=\left(\frac{c}{m\pi}\right)^2\mathbf f(\xi)^*\mathcal B_m(\xi)\mathbf f(\xi),\ \xi\in \R.
\]
Combining the last equation with the
 Parseval's relation
\begin{align}
\label{eq_lll1}
\int_0^1|b(\xi,t)|^2d\xi=\sum_{k\in\Z}\left|f_t\left(\frac{m \pi}{c}k\right)\right|^2.
\end{align}
yields the desired conclusion. %
\end{proof}

\begin {remark}
\label {vecvsper}
Lemma \ref{lemma_calc} shows that the stability of reconstruction from  spatio-temporal samples
\label{eq_st_samp} is controlled by the condition number of the self-adjoint matrices $\mathcal B_m(\xi)$ in \eqref{matrB}. For symmetric $\phi \in \Phi$ and $m \geq 2$, however,
\begin{align*}
\inf_{\xi \in [0,1]} \lambda_{\mathrm{min}} \big(\mathcal B_m(\xi) \big) = \lambda_{\mathrm{min}} \big(\mathcal B_m(0) \big) = 0,
\end{align*}
which precludes the stable reconstruction of all $f \in PW_c$, see, e.g., \cite{ADK15}.
This adds to our explanation of the phenomenon of blind spots in Subsection \ref{tech}.
We can nonetheless hope to find a large set $\widetilde E \subseteq [0,1]$ such that 
$\lambda_{\mathrm{min}} \big(\mathcal B_m(\xi) \big)\geq \kappa$ for $\xi \in \widetilde E$. Then,
repeating the computation in \eqref{eq_note}, we get
\begin{equation}
\label{eq:restricted}
\begin{split}
\int_0^1 \sum_{k\in\Z}\left|f_t\left(\frac{m \pi}{c}k\right)\right|^2\,\mathrm{d}t
=&
\left(\frac{c}{m\pi}\right)^2\int_0^1\mathbf f(\xi)^*\mathcal B_m(\xi)\mathbf f(\xi)\,\mathrm{d}\xi.
\geq \kappa\left(\frac{c }{m\pi}\right)^2\int_{\tilde E}\|\mathbf f(\xi)\|^2\,\mathrm{d}\xi\\
=&\frac{c\kappa }{2m\pi^2}\int_{E}\|\widehat{f}(\xi)\|^2\,\mathrm{d}\xi
\end{split}
\end{equation}
where $E=\dst\left(\frac{2c}{m}(\tilde E+\Z)\right)\cap[-c,c]$.
\end {remark}

In the following example, we offer some numerics. To simplify the  computations, we represent $\mathcal B_m(\xi)$ in \eqref{matrB} as a Pick matrix (see, e.g., \cite{BT17, FO01}). For $\xi\in [-c,c)$, we write 
$\widehat \phi (\xi )  = e^{-\psi (\xi )},$ so that $\psi \geq 0$ and
$\psi (0)= 0,$ and  obtain for $j, k = 0,\ldots, m-1$,
$$
(\mathcal B_m)_{jk}(\xi ) = 
\int _0^1 \widehat \phi^t\left(\frac{2c}{m}(\xi +j^\prime)\right)\, \widehat  \phi^t
\left(\frac{2c}{m}(\xi +k^\prime)\right)   \, \mathrm{d}t
$$
where the indices  $j^\prime,k^\prime$ are in the set 
\begin{equation}
I_\xi=\left\{ n\in \ZZ :\frac{\xi + n}m \in [-1/2, 1/2)\right\},
\label{eq:ixi}
\end{equation} 
$m$ divides $|j-j^\prime|$ and $|k-k^\prime|$,
 and $j$, $k$, and $\xi$ are not $0$ simultaneously. Thus
\begin{equation}\label{Pickrep}
\begin{split}
(\mathcal B_m)_{jk}(\xi ) & = 
\int _0^1\ e^{-t\left(\psi \left(\frac{2c}{m}(\xi +j^\prime)\right) 
+ \psi\left(\frac{2c}{m}(\xi +k^\prime)\right)\right)} \,\mathrm{d}t \\
&= \left(\psi\left(\frac{2c}{m}(\xi +j^\prime)\right) + \psi\left(\frac{2c}{m}(\xi +k^\prime)\right)\right)\inv
 \,\left(1-e^{-\left(\psi\left(\frac{2c}{m}(\xi +j^\prime)\right) + \psi\left(\frac{2c}{m}(\xi +k^\prime)\right)\right)}  \right)
\end{split}
\end{equation}
Observe that $(\mathcal B_m)_{00}(0) = 1$.

\begin{example}\label{expic1}
Here, we choose $\phi$ to be the Gaussian function, i.e., 
\[
\widehat\phi(\xi) = \widehat\phi_1(\xi) = e^{-{{\sigma^2} \xi^2}}
\]
for various values of $\sigma \not= 0$.
Hence, $\psi(\xi) = \sigma^2\xi^2$, and we get 
\[
  (\mathcal B_m)_{jk}(\xi ) = \frac{m^2}{4c^2\sigma^2}\cdot\frac{1-e^{  - \left(\frac{\sigma^2}{m^2} 
\left({(\xi +j^\prime)^2} + 
                                 {(\xi +k^\prime)^2}\right)\right)}}{(\xi+j^\prime)^2+(\xi+k^\prime)^2}
\]
with $j^\prime$, $k^\prime$,  and $(\mathcal B_m)_{00}(0)$  as above.

In Figure \ref{FIG:variant sigma}, we show the condition numbers of the matrices $\mathcal B_m(\xi)$  with $\xi=0.45$, $c=1/2$, $m\in\{2,3,5\}$, and $\sigma$ varying from $1$ to $200$. 
\end{example}

{
\begin{figure}[ht]
	\centering
	\begin{subfigure}[b]{0.32\linewidth}%
		\includegraphics[width=\textwidth]{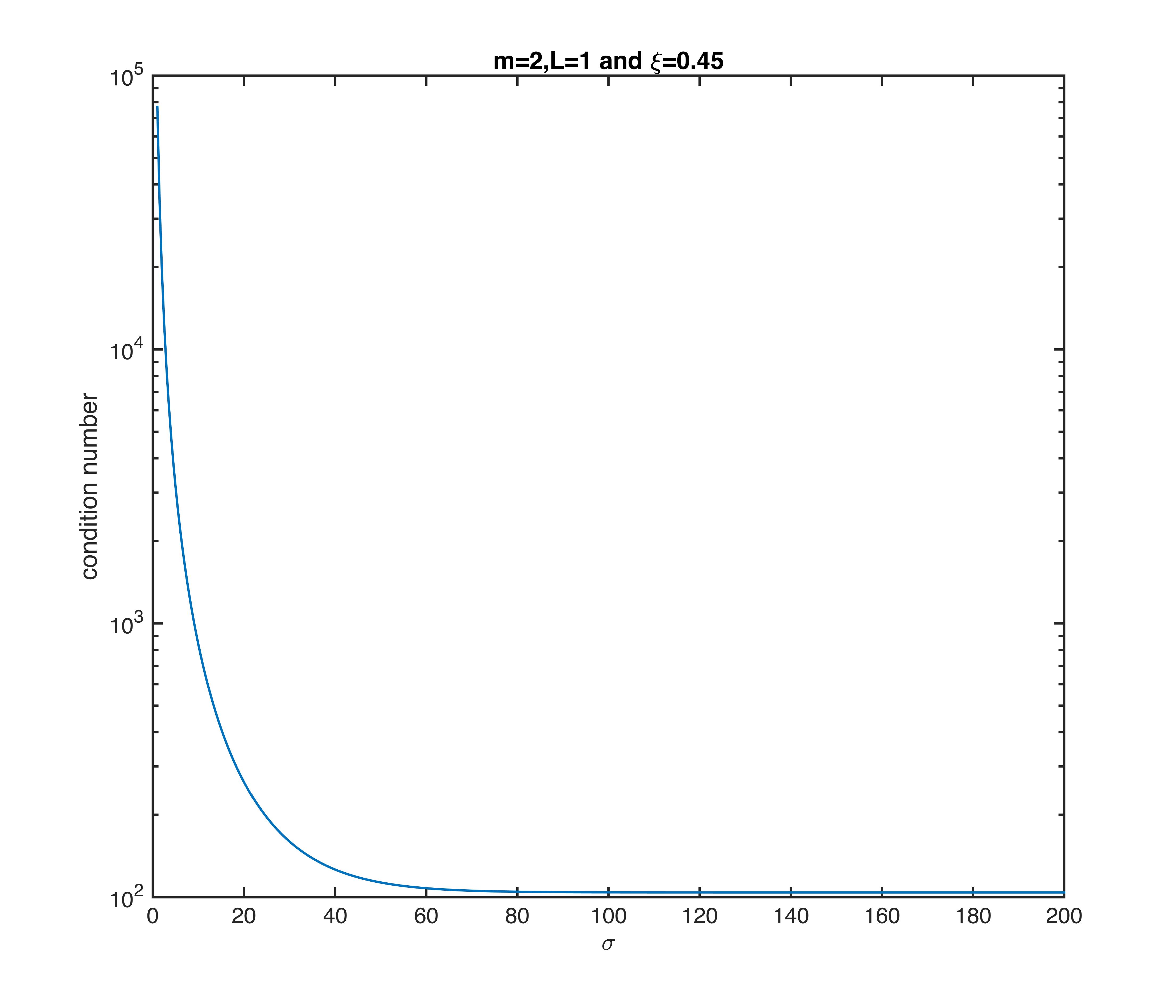}
		\caption{}
		\label{FIG:m2l1s45}
	\end{subfigure}
	\begin{subfigure}[b]{0.32\linewidth}
		\includegraphics[width=\textwidth]{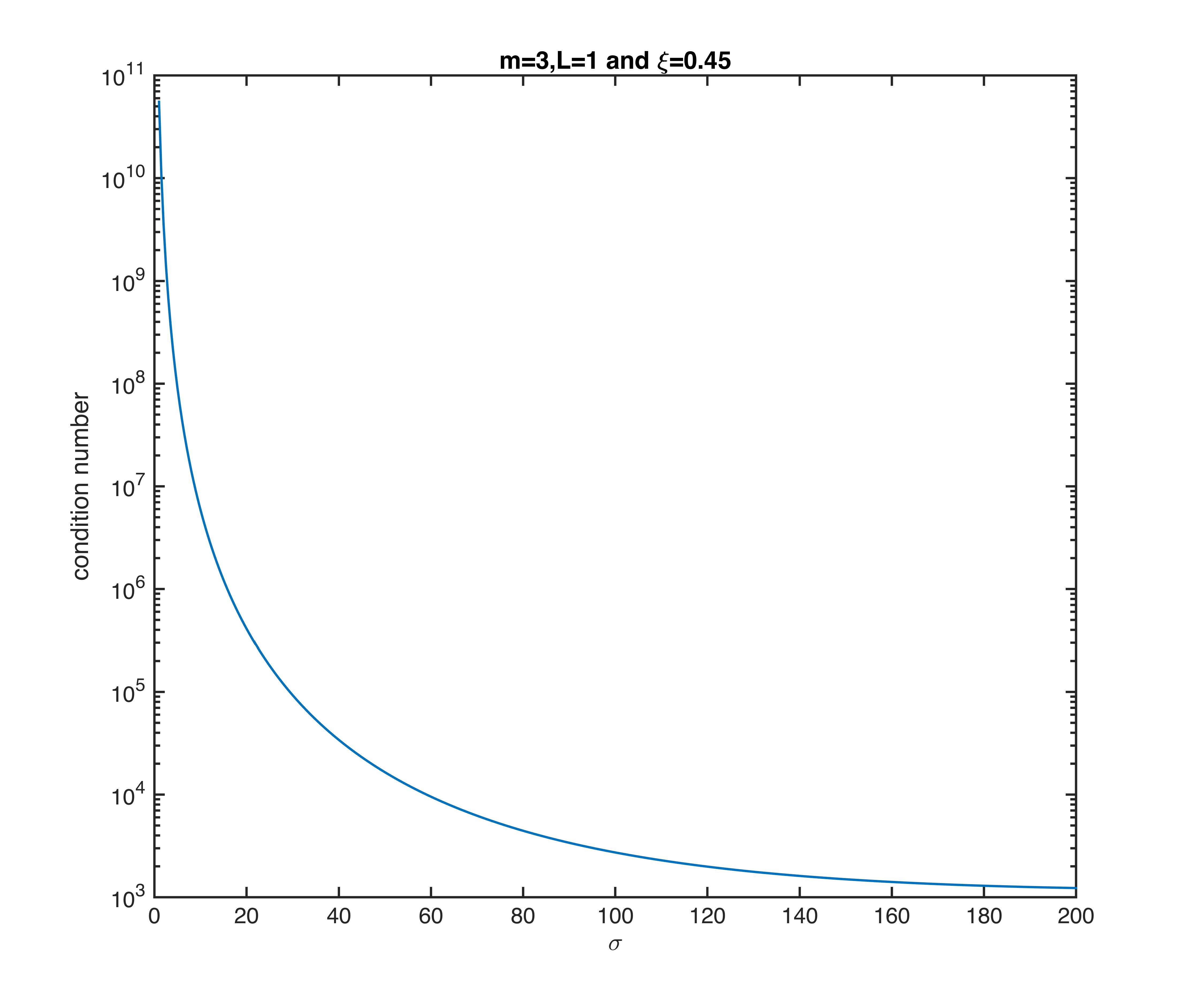}
		\caption{}
		\label{FIG:m3l1s45}
	\end{subfigure}
\begin{subfigure}[b]{0.32\linewidth}
	\includegraphics[width=\textwidth]{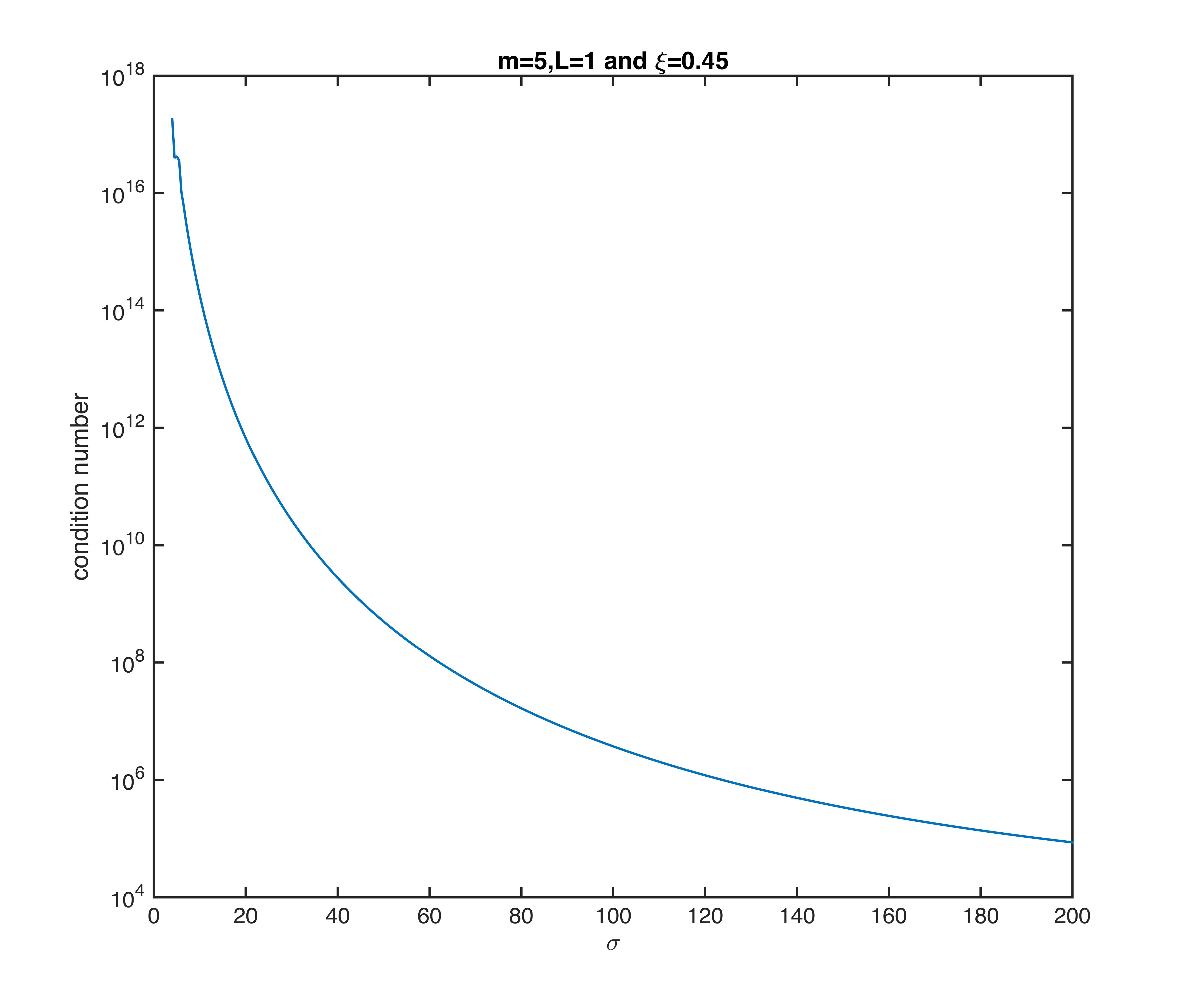}
	\caption{}
	\label{FIG:m5l1s45}
\end{subfigure}
	\caption{Condition numbers of $\mathcal B_m(\xi)$ for $m \in\{2,3,5\}$, $c=1/2$, $\xi=0.45$, and $\sigma\in[1,200]$.}
	\label{FIG:variant sigma} 
\end{figure}

In Figure \ref{FIG:variant xi}, we also show the condition numbers of the matrices $\mathcal B_m(\xi)$. This time, however, still $c=1/2$, the parameter $\sigma$ is fixed to be $200$, whereas the point $\xi$ is allowed to vary from $0.35$ to $0.49$. We still have $m\in\{2,3,5\}$.

\begin{figure}[ht]
	\centering
	\begin{subfigure}[b]{0.32\linewidth}%
		\includegraphics[width=\textwidth]{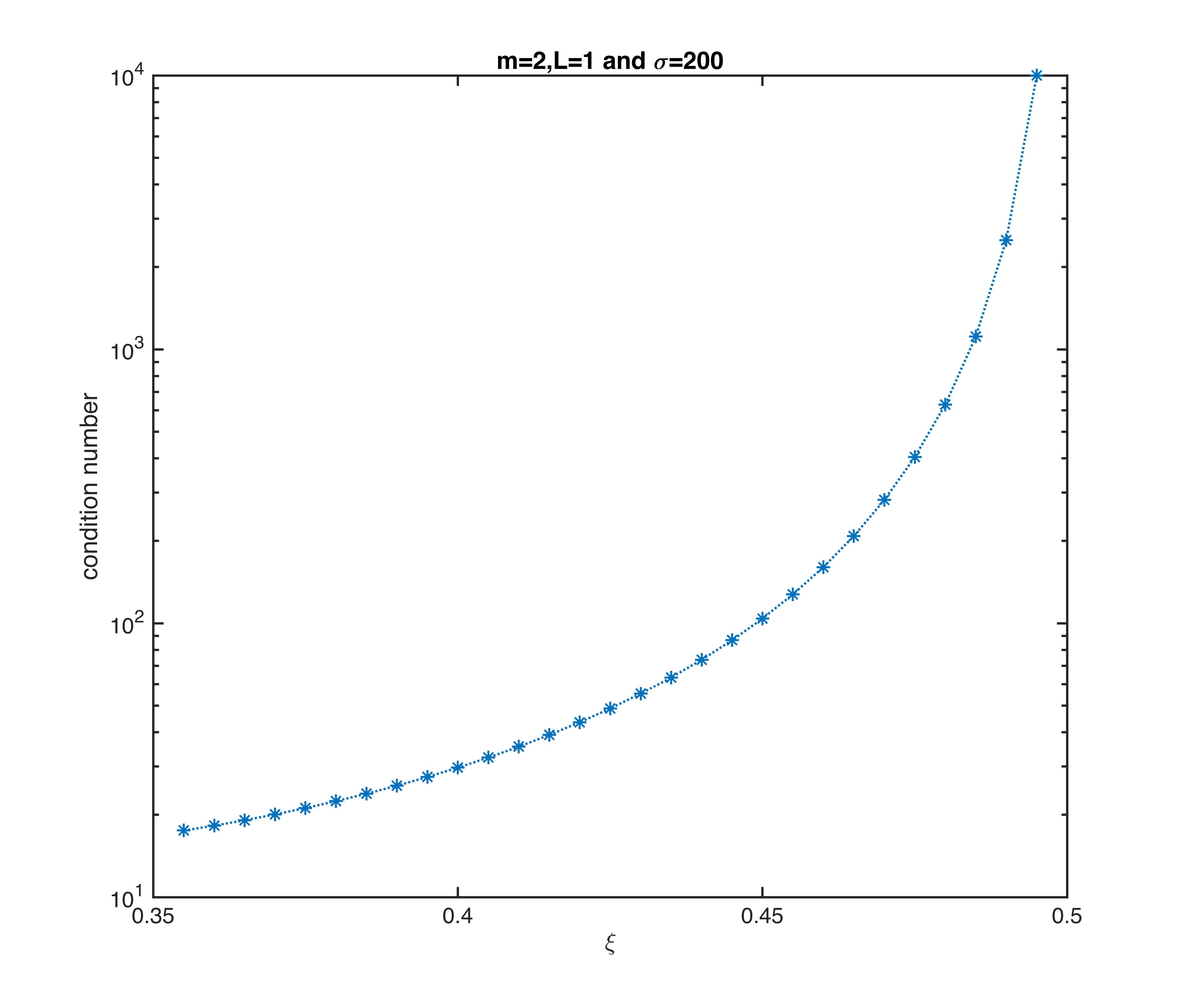}
		\caption{}
		\label{FIG:m2l1sig200}
	\end{subfigure}
	\begin{subfigure}[b]{0.32\linewidth}
		\includegraphics[width=\textwidth]{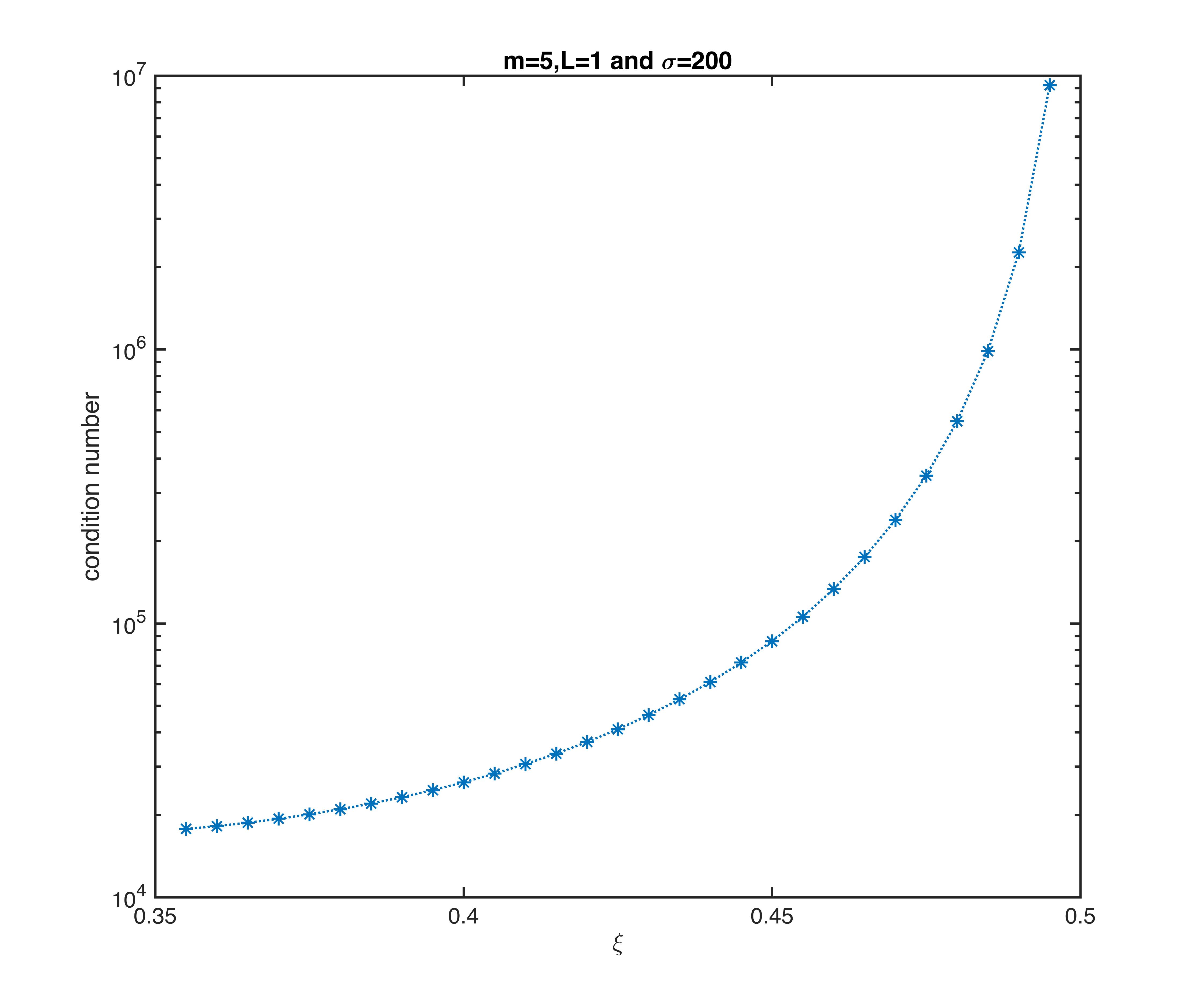}
		\caption{}
		\label{FIG:m3l1sig200}
	\end{subfigure}
	\begin{subfigure}[b]{0.32\linewidth}
          \includegraphics[width=\textwidth]{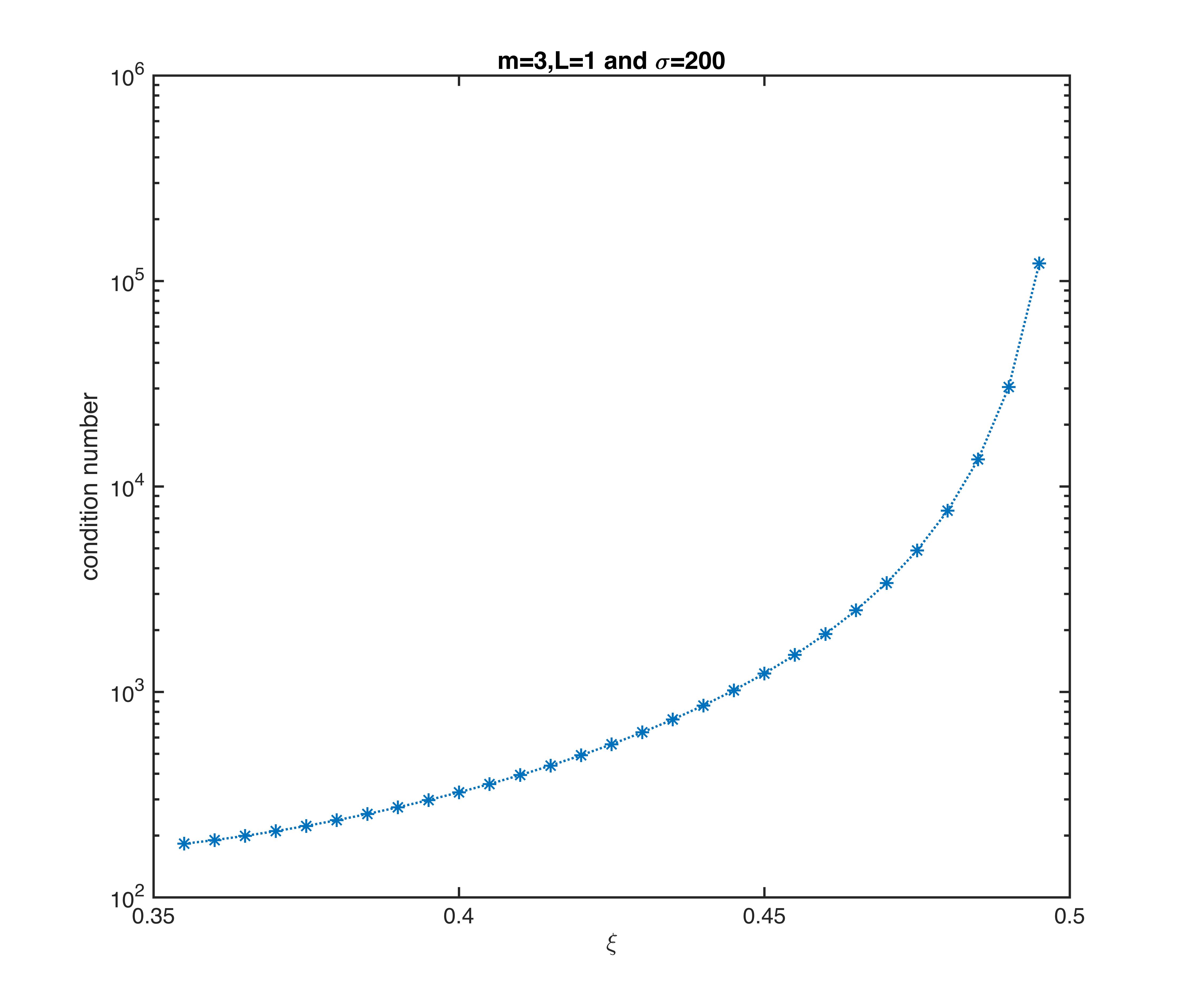}
		\caption{}
		\label{FIG:m5l1sig200}
	\end{subfigure}
	\caption{Condition numbers of $\mathcal B_m(\xi)$ for $m \in\{2,3,5\}$, $c=1/2$, $\sigma=200$, and $\xi\in[0.35,0.49]$.}\label{FIG:variant xi} 
\end{figure}
}

\subsection{Estimating the minimal eigenvalue of the sampled diffusion matrix}
In this subsection, we use 
 Vandermonde matrices
to obtain a lower estimate for the eigenvalue $\lambda^{(m)}_{\min}(\xi)$ of the matrices $\mathcal B_m(\xi)$ in \eqref{matrB}. We also present an upper estimate for  $\lambda^{(m)}_{\min}(\xi)$, which follows from the general theory of Pick matrices \cite{BT17, FO01}.

We begin with the following auxiliary result.
\begin{lemma}\label{lemV}
Let $v_0, v_1, \ldots v_{m-1}$ be $m$ distinct non-zero real numbers and let
$\mathbf{v}=(v_0,\ldots,v_{m-1})$. For $k\in\N$, define a function $\Psi_k:\R\to\R$ by
 $\Psi_k(t)=\tfrac{1-t^2}{1-t^{2/k}}$ if $t\neq1$ and
$\psi_k(1)=k$.
For $j=0,\ldots,m-1$, define
$$
\sigma_j^2=\sum_{k=0}^{m-1}v_j^{2k}= \Psi_m(v_j^m) = 
\begin{cases}m&\mbox{if }v_j=1\\
\frac{1-v_j^{2m}}{1-v_j^2}&\mbox{otherwise}\end{cases}.
$$
Let $\sigma=\left(\sum_{j=0}^{m-1}\sigma_j^2\right)^{1/2}$, $\gamma_- = \min_{j} |v_j| >0$, $\gamma_+=\max_{j} |v_j|$  and let
\[
\alpha= \left(\frac{m-1}{\sigma^2}\right)^{\frac{m-1}{2}}\prod_{0\leq j< k\leq m-1}|v_j-v_k|.
\]
For $N\in \mathbb N$, let $W_N$ be the $(mN)\times m$ Vandermonde matrix associated to 
$\mathbf{v}_N = (v_0^{\frac{1}{N}}, v_1^{\frac{1}{N}}, \ldots v_{m-1}^{\frac{1}{N}})$,  i.e.,
$$
W_N=\left[v_j^{\frac{i-1}{N}}\right]_{1\leq i\leq mN,0\leq j\leq m-1}.
$$
Then for each $x\in\CC^m$, we have
\[
\alpha^2 \Psi_N(\gamma_-)\|x\|^2 \le \|W_N x\|^2 \le 
 \sigma^2\Psi_N(\gamma_+)\|x\|^2.
\]
\end{lemma}

\begin{proof}
let $V$ be the $m\times m$ Vandermonde matrix associated to $\mathbf{v}$:
$$
V=[v_j^{i}]_{0\leq i\leq m-1,0\leq j\leq m-1}.
$$
Note that the Frobenius norm of $V$ and its determinant are given by
\[
\norm{V}_F=\sigma \quad\mbox{and}\quad
|\det V|=\prod_{0\leq j< k\leq m-1}|v_j-v_k|.
\]
Recall from \cite{YG97} an estimate for the minimal singular value of an $m\times m$ matrix $A$: 
\begin{equation}\label{minest}
\sigma_{\min}(A) \ge \left(\frac{m-1}{\|A\|^2_F}\right)^{(m-1)/2}|\det A|.
\end{equation}
Specifying this to $V$ we get $\sigma_{\min}(V)\geq\alpha$. As $\norm{V}\leq\norm{V}_F$,
it follows that, for all $x\in \CC^m$,
\begin{equation}
\alpha^2\|x\|^2\leq \|Vx\|^2\leq \sigma^2\|x\|^2.
\label{eq:frob}
\end{equation}

Let $D_N$ be the diagonal matrix with $\mathbf v_{N}$ on the main diagonal. Since
\[
\|W_N x\|^2 = \langle W_N^* W_N x, x\rangle = \sum_{\ell=0}^{N-1} \langle (D_N^\ell)^* V^*VD_N^\ell x, x\rangle = \sum_{\ell=0}^{N-1}\|VD_N^\ell x\|^2, 
\]
we deduce from \eqref{eq:frob} that
\[
\sum_{\ell=0}^{N-1}\alpha^2 \|D_N^\ell x\|^2\le \|W_N x\|^2 \le
\sum_{\ell=0}^{N-1}\sigma^2 \|D_N^\ell x\|^2.
\]
Moreover, we have $\gamma_-^{\frac{2\ell}N}\|x\|^2\le \|D_N^\ell x\|^2 \le  \gamma_+^{\frac{2\ell}N}\|x\|^2$ by definition of $D_N$. The conclusion now follows by summing the two geometric sequences. 
\end{proof}

Note that the function $\Psi_N$ is increasing on $(0,+\infty)$ and that, for $t\neq1$, $t>0$
\begin{equation}\label{limpsi}
\lim_{N\to\infty}\frac{1}{N}\Psi_N(t)=\frac{1-t^2}{\lim_{N\to\infty}N(1-e^{2\ln t/N})}=\abs{\frac{1-t^2}{2\ln t}}.
\end{equation}

\begin{coro}\label{corV}
With the notation of Lemma \ref{lemV}, assume further that $0<\nu\leq v_j\leq 1$ and $m\geq 2$.
Let

\begin{equation}\label{alphat}
\widetilde\alpha=e^{-1/2}m^{-\frac{m-1}{2}}\prod_{0\leq j< k\leq m-1}|v_j-v_k|.
\end{equation}

Then for each $x\in\CC^m$, we have
\[
\widetilde\alpha^2 \Psi_N(\nu)\|x\|^2 \le \|W_N x\|^2 \le 
m^2N\|x\|^2.
\]
\end{coro}

\begin{proof}
Indeed, $\nu\leq\gamma_-\leq\gamma_+\leq 1$ so $\Psi_N(\nu)\leq\Psi_N(\gamma_-)$
and $\Psi_N(\gamma_+)\leq\Psi_N(1)=N$.

Further, %
since $v_j\le 1$, $\sigma^2\leq m^2$.
Moreover, the derivative of $\left(\frac{t-1}{t}\right)^{(t-1)/2}=\left(1-\frac{1}{t}\right)^{(t-1)/2}$ is
$$
\frac{1}{2}\left(1-\frac{1}{t}\right)^{(t-1)/2}\left(\frac{1}{t}+\ln\left(1-\frac{1}{t}\right)\right)\leq 0
$$
for $t\geq 1$. Thus,
$$
\left(\frac{m-1}{m}\right)^{(m-1)/2}\geq\lim_{t\to+\infty}\exp\left[\frac{t-1}{2}\ln\left(1-\frac{1}{t}\right)\right]=e^{-1/2}.
$$
It follows that $\alpha$ in the statement of Lemma \ref{lemV} satisfies
$$
\alpha\geq\frac{\prod_{0\leq j< k\leq m-1}|v_j-v_k|}{\sqrt{e}m^{(m-1)/2}},
$$
and the result is established.
\end{proof}

\begin{prop}
\label{cor:Sec5}
Let $ \phi \in \Phi$. Define
$$
\newdelta(\xi)=
\prod_{0\leq j< k\leq m-1}\abs{\widehat\phi_p\left(\frac{2c}{m}(\xi+j)\right)
-\widehat\phi_p\left(\frac{2c}{m}(\xi+k)\right)}.
$$
Then, for each $x\in\CC^m$, we have
\[
\frac{1}{2e m^{m^2}}\newdelta(\xi)^2 \cdot\frac{1-\newmphi^{2/m}}{|\ln \newmphi|}\|x\|^2\le
\langle\mathcal B_m(\xi) x, x \rangle \le m\|x\|^2.
\]
\end{prop}

\begin{proof}
We fix $\xi$ and apply Corollary \ref{corV} to 
$v_j =\dst(\widehat\phi)_p\left(\frac{2c}{m}(\xi+j)\right)^{\frac{1}{m}}$. With $\widetilde\alpha$ given by \eqref{alphat},
$$
\widetilde\alpha=
e^{-1/2}m^{-\frac{m-1}{2}}
\prod_{0\leq j< k\leq m-1}\abs{\widehat\phi_p\left(\frac{2c}{m}(\xi+j)\right)^{1/m}
-\widehat\phi_p\left(\frac{2c}{m}(\xi+k)\right)^{1/m}},
$$
we get
$$
\widetilde\alpha^2\Psi_N(\newmphi^{1/m})\|x\|^2\leq\|W_Nx\|^2\leq m^2 N \|x\|^2.
$$

On the other hand,
$\frac{1}{mN} W_N^*W_N$
equals  the left-end $mN$-term Riemann sum for the integral defining $\mathcal B_m(\xi)$.
It follows that
\[
\langle\mathcal B_m(\xi) x, x \rangle = \lim_{N\to \infty}\frac{1}{mN}
\langle W_N^*W_Nx, x \rangle=\lim_{N\to \infty}\frac{1}{mN}
\|W_Nx\|^2.
\]
Using \eqref{limpsi}, we get
\[
\widetilde\alpha^2\frac{1-\newmphi^{2/m}}{2m|\ln \newmphi|}\|x\|^2\leq 
\langle\mathcal B_m(\xi) x, x \rangle\leq m\|x\|^2.
\]
Finally, note that if $0<a,b\leq 1$, using the mean value theorem, there is an $\eta\in(a,b)$ such that
$$
|a^{1/m}-b^{1/m}|=\frac{1}{m}|a-b|\eta^{-1+1/m}\geq \frac{1}{m}|a-b|.
$$
Therefore
\begin{eqnarray*}
\widetilde\alpha&=&
e^{-1/2}m^{-\frac{m-1}{2}}
\prod_{0\leq j< k\leq m-1}\abs{\widehat\phi_p\left(\frac{2c}{m}(\xi+j)\right)^{1/m}
-\widehat\phi_p\left(\frac{2c}{m}(\xi+k)\right)^{1/m}}\\
&\geq&e^{-1/2}m^{-\frac{m-1}{2}-\frac{m(m-1)}{2}}\Delta(\xi)
=e^{-1/2}m^{-\frac{m^2-1}{2}}\Delta(\xi)
\end{eqnarray*}
establishing the postulated estimates.
\end{proof}

For an upper estimate of the minimal eigenvalue $\lambda^{(m)}_{\min}(\xi)$ we use the estimates of the
singular values of Pick matrices by Beckerman-Townsend \cite{BT17}.
For $p_j\in \CC $, $ j = 1 , \dots , m$,  and $0<a\leq x_1 < x_2 < \dots < x_m \leq
b$ let 
\begin{equation}\label{eq:add44}
(P_m)_{jk} = \frac{p_j+ p_k}{x_j+ x_k},  \qquad j,k = 1 , \dots , m,
\end{equation}
be the corresponding Pick matrix. Then the smallest singular value $s_{\min}$ of
$P_m$ is bounded above by
\begin{equation}
  \label{eq:add4}
s_{\min} \leq \min\left\{1, 4 \left[\exp \left( \frac{\pi^2}{2\ln \left(\frac{4b}{a}\right)} \right)\right]
^{-2\lfloor m/2 \rfloor}\right\} s_{\max},
\end{equation}
where $s_{\max}$ is the largest singular value.

If $(\widetilde{ P}_m)_{jk} = \frac{1-c_jc_k}{x_j+x_k}$, %
then $\widetilde{P}_m$ is related to a Pick matrix of the form
\eqref{eq:add44} via the diagonal matrix $D = \mathrm{diag}
(1+c_j)$:
$$
\frac{1}{2} D \inv \widetilde{P}_m D \inv = P_m
$$
with $p_j = \frac{1-c_j}{1+c_j}$, $c_j\neq-1$.

In our case, see \eqref{Pickrep}, $x_j= \psi \left(\frac{2c}{m}(\xi +j)\right)$ and $c_j = e^{-x_j} \in
(0,1]$, so $\mathrm{Id}\leq D \leq 2\mathrm{Id}$ and  the
singular values of $\cB _m (\xi )$  and the corresponding Pick matrix  $P_m$ 
differ at most 
by a factor $4$. Therefore, \eqref{eq:add4} holds with  
$a (\xi ) = \min \left\{\psi \left(\frac{2c}{m}(\xi +k)\right) : k\in I_\xi \right\}$ and  
$b (\xi ) = \max \left\{\psi \left(\frac{2c}{m}(\xi +k)\right) : k\in I_\xi \right\}$,
$I_\xi$ defined in \eqref{eq:ixi},
and an additional factor $4$ provided that $a(\xi) \neq 0$.

For  our main  examples, we have $\psi (\xi ) = |\xi |^\alpha $, $\alpha> 0$. This yields 
$$
b(\xi ) \leq c^\alpha
\quad\mbox{and}\quad
a(\xi ) = \min \left\{\left|\frac{2c}{m}(\xi -k)\right|^\alpha : 
\frac{2c}{m}|\xi - k|\leq c,\ |\xi|\leq\frac{1}{2}\right\} = \left(\frac{2c}{m}|\xi|\right)^\alpha
$$
So for the smallest singular value of $\cB _m(\xi )$ we obtain the
estimate
\begin{equation}  \label{eq:add5}
\begin{split}
\lambda^{(m)}_{\min}(\xi)&\leq 4^2\left[ \exp \left( \frac{\pi^2}{2\ln
  4\left(\frac{m}{2|\xi|}\right)^\alpha} \right) \right]
         ^{-2\lfloor m/2 \rfloor} m\\
       &\leq 16 m \, \exp \left(-\frac{(m-1)\pi ^2}{\ln 16 + 2\alpha
         \ln\frac{m}{2|\xi|}} \right). 
\end{split}
\end{equation}

Observe that the Beckerman-Townsend estimate \eqref{eq:add4}   holds for \emph{all}
Pick matrices with the same values for $a=\min x_j$ and $b=\max x_j$
and is completely 
independent of the particular distribution of the $x_j$. Regardless, it shows that the condition
number grows nearly exponentially with $m$, establishing limitations on how well the space-time trade-off can work numerically.
Of course, the condition number
  may be much
worse if two values $x_j$ and $x_{j+1}$ are
close together (if $x_j = x_{j+1}$, then $P_m$ is 
singular). Thus,  \eqref{eq:add5} is an optimistic upper estimate for
$\lambda^{(m)}_{\min}(\xi)$. By comparison, our lower estimate in Proposition \ref{cor:Sec5} depends
crucially on the distribution of the parameters $x_j$ and is much harder to obtain. It does, however, establish an upper bound on the 
condition number and, thus, shows that the space-time trade-off may be useful. A precise result is formulated in the following subsection.

\subsection{Partial recoverability}

 \begin{theorem}\label{sufconT}
Let $ \phi \in \Phi$, $m\geq 2$ an integer and $\widetilde E\subseteq I = [0,1] $ be a compact set.
Assume that there exists $\delta>0$ such that,
for every $0\leq j<k\leq m-1$ and %
every $\xi \in\widetilde E$
\[
\abs{\widehat\phi_p\left(\frac{2c}{m}(\xi+ j)\right)-\widehat\phi_p\left(\frac{2c}{m}(\xi+k)\right)}\geq\delta.
\]
Let $E=\dst\left(\frac{2c}{m}(\tilde E+\Z)\right)\cap[-c,c]$. Then for any $f\in PW_c$, the function 
$\widehat f\un_{ E}$ can be recovered from the samples
\begin{equation}\label{sampM}
\mathcal M = \left\{f_t\left(\frac{m \pi}{c}k\right): k \in \mathbb{Z}, 0 \leq t \leq 1\right\}
\end{equation}
in a stable way. Moreover, we have
\begin{equation}\label{fest}
A\norm{\widehat f\un_{ E}}^2%
 \le \int_0^1 \sum_{k\in\ZZ}\left|f_t\left(\frac{m \pi}{c}k\right)\right|^2\,\mathrm{d}t \le\frac{c}{2\pi^2}\|\widehat f\|^2,
\end{equation}
where
	\[A=\frac{c }{4e\pi^2}\frac{\delta^{m(m-1)}}{m^{1+m^2}}
	\frac{\newmphi^{2/m}-1}{\ln \newmphi}.
	\]
\end{theorem}

\begin{proof} Recall from \eqref{eq:lister} that we need to estimate
\[
\int_0^1 \sum_{k\in\ZZ}\left|f_t\left(\frac{m \pi}{c}k\right)\right|^2\,\mathrm{d}t=\left(\frac{c}{m\pi}\right)^2\int_0^1\mathbf f(\xi)^*\mathcal B_m(\xi)\mathbf f(\xi)\,\mathrm{d}\xi.
\]
The upper bound follows directly from Proposition \ref{cor:Sec5}, and \eqref{eq_note}:
$$
\int_0^1\mathbf f(\xi)^*\mathcal B_m(\xi)\mathbf f(\xi)\,\mathrm{d}\xi
\leq m\int_0^1\|\mathbf f(\xi)\|^2\,\mathrm{d}\xi
=\frac{m^2}{2c}\norm{\widehat f}^2.
$$

Let us now prove the lower bound using \eqref{eq:restricted}.
First, $\Delta_m(\xi)\geq \delta^{\frac{m(m-1)}{2}}$.
It follows from Proposition \ref{cor:Sec5} that, if $\xi \in \widetilde E$ then
$$
\mathbf f(\xi)^*\mathcal B_m(\xi)\mathbf f(\xi)
\geq \frac{\newmphi^{2/m}-1}{2e m^{m^2}\ln \newmphi} \delta^{m(m-1)}
\|\mathbf f(\xi)\|^2.
$$
Taking $\dst\kappa=\frac{\newmphi^{2/m}-1}{2e m^{m^2}\ln \newmphi} \delta^{m(m-1)}$
in \eqref{eq:restricted} gives the result.
\end{proof}

\begin{remark}
The condition number implied by the above theorem is not the best possible one can obtain through this method.
For instance, a better estimate for the
$\sigma_{\min}$ of a Vandermonde matrix may be used in place of  \eqref{minest}.

However, the method will always lead to a deteriorating estimate of the condition number as $m$ increases. This follows from the Beckerman-Townsend
estimate \eqref{eq:add4} we discussed in the previous subsection.

\end{remark}

\begin{coro}
\label{coro:main}
Assume that $\phi\in\Phi$, $\widehat\phi$ is even and strictly decreasing on $\R_+$, and
$m\geq 2$ is an integer. 
Given $\eta \in (0,\frac1{4})$, let $\widetilde E = [-\frac 1 2 + \eta, -\eta]\cup[\eta, \frac 1 2 -\eta]$ and 
$E=\dst\left(\frac{2c}{m}(\tilde E+\Z)\right)\cap[-c,c]$.
Then there exists
$A>0$ such that, for any $f\in PW_c$, 
\[
A\norm{\widehat f\un_{ E}}^2 \le 
\int_0^1 \sum_{k\in\ZZ}\left|f_t\left(\frac{m \pi}{c}k\right)\right|^2\,\mathrm{d}t \le \|\widehat f\|^2.
\]
\end{coro}

\begin{proof}
We look into the main condition of Theorem \ref{sufconT}: {\em there exists $\delta>0$ such that,
for every $0\leq j<k\leq m-1$ and every $\xi\in \widetilde E$}
\begin{equation}
\label{eq:condthm52}
\abs{\widehat\phi_p\left(\frac{2c}{m}(\xi+j)\right)-\widehat\phi_p\left(\frac{2c}{m}(\xi+k)\right)}\geq\delta.
\end{equation}
For a general $\phi\in\Phi$, the function $\widehat\phi$ is continuous and, therefore, $\widehat\phi_p$ is continuous, except possibly on $\displaystyle c+2c\Z$
where a jump discontinuity occurs if $\widehat\phi(-c)\neq\widehat\phi(c)$. Under current assumptions, however, $\widehat\phi$ is even and, therefore $\widehat\phi_p$ is continuous everywhere.

For $0\leq\ell\leq m-1$ and $\xi\in I$, we have $\displaystyle -\frac 1 {2m} \le \frac{\xi+ \ell}{m}\le1-  \frac 1 {2m} $ and 
$$
\widehat\phi_p\left(\frac{2c}{m}(\xi+ \ell)\right)
=\begin{cases}\dst\widehat\phi\left(\frac{2c}{m}(\xi+ \ell)\right)&\mbox{if }\dst
\frac{\xi+ \ell}{m}< 1/2\\
\dst\widehat\phi\left(\frac{2c}{m}(\xi+ \ell-m)\right)&\mbox{if }\dst\frac{\xi+ \ell}{m}\ge 1/2
\end{cases}.
$$

Thus, the condition  of Theorem \ref{sufconT} would be satisfied with $\widetilde E=I$ if $|\widehat\phi|$ 
were one-to-one on $I$, that is, either strictly decreasing or strictly increasing.
However, $\widehat\phi$ is even and strictly decreasing on $\R_+$, so that $\widehat\phi_p$ is continuous, strictly 
decreasing on $[0,c]$ and strictly increasing on $[-c,0]$. It follows that \eqref{eq:condthm52} may only fail in 
small intervals around the  points $\xi \in I$ 
where $\dst\widehat\phi_p\left(\frac{2c}{m}(\xi+j)\right)-\widehat\phi_p\left(\frac{2c}{m}(\xi+ k)\right) = 0$
for some $j,k\in\Z$. Such points must satisfy 
\[
\frac{\xi+ j}{m} = 1-\frac{\xi+\ell}{m},\ 0\le j \le \frac {m-1}{2}  < \ell\le m-1.
\]
Thus, we need $\xi = \frac{1}{2}(m-j-\ell)$, i.e.~$\xi\in\{0,\pm \frac{1}{2} \}$. In view of the continuity of 
$\widehat\phi_p$, it follows that there exists $\eta > 0$ such that \eqref{eq:condthm52} holds for 
$\xi\in \widetilde E = [-\frac 1 2 + \eta, -\eta]\cup[\eta, \frac 1 2 -\eta]$. It remains to observe that with 
any given $\eta  \in (0,\frac1{4})$ inequality \eqref{eq:condthm52} will hold for $\delta$ sufficiently small. 
\end{proof}

\subsection{Explicit quantitative estimates for the Gaussian}\

{%
To obtain explicit estimates, we need to establish a precise relation between $\eta$ and $\delta$ in the proof of Corollary \ref{coro:main}. In other words,
we need to estimate $\min\limits_{\xi \in\widetilde E} \psi(\xi)$,  where, as above, $\widetilde E = [-\frac 1 2 + \eta, -\eta]\cup[\eta, \frac 1 2 -\eta]$, $\eta \in(0, \frac1{4})$, and the function $\psi$ is given by
$$
\omega(\xi) = \min_{j,k\in\Z} \abs{\widehat\phi_p\left(\frac{2c}{m}(\xi+ j)\right)
-\widehat\phi_p\left(\frac{2c}{m}(\xi+k)\right)}.
$$
\begin{lemma}\label{Liplemma}
Let $E$ and $\widetilde E$ be as in  Corollary \ref{coro:main}.
Assume that the kernel $\phi\in\Phi$ is such that $\widehat\phi$ is differentiable on $E$ and 
$$
\min_{\xi \in  E} \abs{\widehat\phi^\prime(\xi)}\ \ge R.
$$ 
Then
\[
 \min\limits_{\xi\in\widetilde E}\omega(\xi) \ge\frac{4cR\eta}{m}. 
\]
\end{lemma}
\begin{proof}
Observe that %
\[
 \min_{\xi\in\widetilde E}\min_{j,k\in\Z} \abs{\left|\frac{2c}{m}(\xi+ j)\right|-\left|\frac{2c}{m}(\xi+ k)\right|} = \frac{2c}{m}2\eta.
\]
With this, the assertion of the lemma follows immediately from the mean value theorem.
\end{proof}

The above observation leads to the following explicit estimate for the Gaussian kernel.

\begin{prop}
\label{prop:gauss}
Let $\widehat\phi(\xi)=e^{-\sigma^2\xi^2}$, $\sigma\not=0$, and $m\geq 2$ be an integer. Given $\eta  \in (0,\frac1{4})$, let $\widetilde E = [-\frac 1 2 + \eta, -\eta]\cup[\eta, \frac 1 2 -\eta]$ and 
$E=\dst\left(\frac{2c}{m}(\tilde E+\Z)\right)\cap[-c,c]$. 
Then,  for any $f\in PW_c$, we have
\begin{equation}
\label{eq_ppp}
A\norm{\widehat f\un_{E}}^2 \le \int_0^1 \sum_{k\in\ZZ}\left|f_t\left(\frac{m\pi}{c}k\right)\right|^2\,\mathrm{d}t \le \|\widehat f\|^2,
\end{equation}
where
\begin{equation}\label{estA}
A= \frac{c}{2e\pi^2(2(\sigma c)^2+m)}\frac{(4cR\eta)^{m(m-1)}}{{m^{1-m+2m^2}}}
\quad\mbox{with}\quad
R = 2\sigma^2\min\left\{\eta e^{-(\sigma\eta)^2}, ce^{-(\sigma c)^2}\right\}.
\end{equation}
\end{prop}

\begin{proof}
Observe that Lemma \ref{Liplemma} applies with $R$ given by \eqref{estA}.
It remains to apply Theorem \ref{sufconT} with $\newmphi = e^{-(\sigma c)^2}$ and $\delta = {4cR\eta/m}$.
We deduce that  \eqref{eq_ppp} holds with
\begin{align*}
A &=\frac{c}{4e\pi^2}\frac{\delta^{m(m-1)}}{ m^{1+m^2}}
\frac{\newmphi^{2/m}-1}{\ln \newmphi}
=\frac{c}{2e\pi^2}\frac{1-e^{-\tfrac{2(\sigma c)^2}{m}}}{\tfrac{2(\sigma c)^2}{m}}\cdot\frac{(4cR\eta)^{m(m-1)}}{{m^{2-m+2m^2}}}.
\end{align*}
Using $\dst\frac{1-e^{-t}}{t}\geq\frac{1}{t+1}$, we obtain the claimed bound.
\end{proof}

We remark that the estimate in the above proposition is quite
pessimistic. Our numerical experiments showed that the true bound may
be much  better.
}

\section{Remez-Tur\'an Property and Fixing the Blind Spots}\label{rtpbs}
In Theorem \ref{sufconT}, the main issue is that the lower bound is only in terms of $\norm{\widehat f\un_{E}}$ 
and not $\norm{\widehat f}$ so that stability is not obtained. %
In this section, we consider a certain class of subsets of $PW_c$ for which Theorem \ref{sufconT} does lead to stable reconstruction.

\subsection{Remez-Tur\'an Property}

\begin{definition}
Let $V\subset PW_c$ and write $\widehat V=\{\widehat f\,:\ f\in V\}\subset L^2([-c,c])$.
We will say that $\widehat V$ has the Remez-Tur\'an property if, for every $E\subset [-c,c]$ of positive Lebesgue measure,
there exists $C=C(E,V)$ such that, for every $f\in V$, 
\begin{equation}\label{eq:tr}
\norm{\widehat f\un_E}_2\geq C\norm{\widehat f\un_{[-c,c]}}_2.
\end{equation}
\end{definition}
When $V$ is a finite dimensional subspace of $PW_c$ such that $\widehat V$ consists of analytic functions 
(restricted to $I$), then $\widehat V$ has the Remez-Tur\'an property since $\norm{\widehat f\un_E}_2$
is then a norm on $V$ which, by finite dimensionality of $V$, is equivalent to $\norm{\widehat f\un_{[-c,c]}}_2$.
However, the previous argument does not provide any quantitative estimate on the constant $C(E,V)$.
Let us start with two fundamental examples of spaces that have the Remez-Tur\'an property, and for which quantitative estimates are known.

\subsection{Fourier polynomials}
Let $V_N$ be given by \eqref{eq_modc}, so that $\widehat V_N=\{P\un_{[-c,c]}, P\in\C_N[x]\}$ is the space of polynomials of degree at most $N$, restricted to $I$. The quantitative form of the Remez-Tur\'an property for $\widehat V_N$ is then known as the Remez Inequality \cite{BE95}: for every polynomial of degree at most $N$,
\begin{equation}
\label{eq:remez}
\norm{P\un_{[-c,c]}}_2%
\leq \left(\frac{8c}{|E|}\right)^{N+1/2}\norm{P\un_E}_2. %
\end{equation}

\subsection{Sparse sinc translates with free nodes}
Let $V_N$ be given by \eqref{eq_modb}, so that $\widehat V_N=\dst\left\{P\un_{[-c,c]}: P(\xi)=\sum_{n=1}^Nc_ne^{2i\pi\lambda_n\xi}\right\}$. Recall that $\widehat V_N$ is not a linear subspace. The fact that $\widehat V_N$ has the Remez-Tur\'an property is a deep result of Nazarov \cite{Na}: for every exponential polynomial of order at most $N$, {\it i.e.} every $P$ of the form $P(\xi)=\sum_{n=1}^Nc_ne^{2i\pi\lambda_n\xi}$ one has
\begin{equation}
\label{eq:remezpol}
\norm{P\un_{[-c,c]}}%
\leq \left(\frac{\gamma c}{|E|}\right)^{N+1/2}\norm{P\un_E},
\end{equation}
where $\gamma$ is an absolute constant.

\subsection{Prolate spheroidal wave functions (PSWF)}
The Prolate spheroidal wave functions (PSWFs)
denoted by $(\psi_{n,c}(\cdot))_{n\geq 0}$, are defined as the bounded eigenfunctions  of
the Sturm-Liouville differential operator $\mathcal L_c,$ defined on
$C^2([-1,1]),$ by
\begin{equation}
\label{eq1.0}
\mathcal{L}_c(\psi)=-(1-x^2)\frac{d^2\psi}{d\, x^2}+2 x\frac{d\psi}{d\,x}+c^2x^2\psi.
\end{equation}
They are also  the eigenfunctions of the finite Fourier transform $\mathcal F_c$, as well as the ones of the operator
$\displaystyle \mathcal Q_c= \frac{c}{2\pi} \mathcal F^*_c \mathcal F_c ,$ which are defined  on $L^2([-1,1])$ by
\begin{equation}
\mathcal{F}_c(f)(x)= \int_{-1}^1
e^{i\, c\, x\, y} f(y)\,\mbox{d}y,
\quad\mbox{and}\quad
\mathcal{Q}_c(f)(x)=\int_{-1}^1 \frac{\sin(c(x-y))}{\pi (x-y)} f(y)\,\mbox{d}y.
\end{equation}
They are normalized so that $\|\psi_{n,c}\|_{L^2([-1,1])}=1$  and $\psi_{n,c}(1)>0$.
We call  $(\chi_n(c))_{n\geq 0}$ the corresponding eigenvalues of $\mathcal L_c$, $\mu_n(c)$ the eigenvalues of 
$\mathcal{F}_c$ 
\begin{equation}
\label{eq:eigen}
\mu_n(c)\psi_{n,c}(x)=\int_{-1}^1\psi_{n,c}(y)e^{-icxy}\,\mbox{d}y,\ x\in [-1,1].
\end{equation}
and $\lambda_n(c)$ the ones of $\mathcal{Q}_c$ which are arranged in decreasing order. They are related by
$$
\lambda_n(c)=\dst\frac{c}{2\pi}|\mu_n(c)|^2.
$$
A well known property is then that $\norm{\psi_{n,c}}_{L^2(\R)}=\frac{1}{\sqrt{\lambda_n(c)}}$.
Further, their Fourier transform is given by
\begin{equation}
\label{eq:prolateFourier}
\widehat{\psi_{n,c}}(\xi)=
\int_{\R}\psi_{n,c}(x)e^{-ix\xi}\,\mbox{d}x
=(-1)^k\frac{2\pi}{c}\frac{\mu_n}{|\mu_n(c)|^2}\psi_{n,c}\left(\frac{\xi}{c}\right)\un_{|\xi|\leq c}
\end{equation}
The crucial commuting property of $\mathcal L_c$ and $\mathcal{Q}_c$ has  been first observed by 
Slepian and co-authors \cite{prolate1}, whose name is closely associated with all properties of PSWFs, 
the spectrum of the operators $\mathcal L_c$ and $Q_c$ and almost time- and band-limited functions.
Among the basic properties of PSWFs, we cite their analytic extension to the whole real
line and their unique properties to form an orthonormal basis of
$L^2([-1,1])$  and an orthogonal basis of $PW_c$.

The prolate spheroidal wave functions admit a good representation in terms
of the orthonormal basis of Legendre polynomials.
In agreement with the standard practice, we will be denoting by $P_k$ the classical Legendre
polynomials, defined by the three-term recursion
$$
P_{k+1}(x) =\frac{2k + 1}{k + 1}x P_k(x) - \frac{k}{k + 1}P_{k-1}(x),
$$
with the initial conditions
$$
P_0(x) = 1, P_1(x) = x.
$$
These polynomials are orthogonal in $L^2([-c,c])$ and are normalized so that
$$
P_k(1)=1\quad\mbox{and}\quad\int_{-1}^1 P_k(x)^2\,\mbox{d}x=\frac{1}{k+1/2}.
$$
We will denote by $P_{k,c}$ the normalized Legendre polynomial
$\dst \widetilde P_{k,c}(x)=\sqrt{\frac{2k+1}{2c}}P_k\left(\frac{x}{c}\right)$ and the $P_{k,c}$'s then form an orthonormal basis of $L^2([-c,c])$.

We start from the following identity relating Bessel functions of the first kind to the finite Fourier transform 
of the Legendre polynomials,  {\it see} \cite{Andrews}: for every $x\in\R$,
\begin{equation}
\label{eq:2.2.9}
\int_{-1}^1 e^{i x y}  P_k(y)\,\mbox{d}y =2i^kj_k(x), \ k\in\N,
\end{equation}
where $j_k$ is the spherical Bessel function defined by 
$j_k(x)=\displaystyle (-x)^k\left(\frac{1}{x}\frac{\mbox{d}}{\mbox{d}x}\right)^k\frac{\sin x}{x}$.
Note that $j_k$ has the same parity as $k$ and recall that, for $x\geq0$, 
$j_k(x)=\sqrt{\frac{\pi}{2x}}J_{k+1/2}(x)$ where $J_\alpha$ is the Bessel function
of the first kind. In particular, from the well-known bound $|J_{\alpha}(x)|\leq \frac{|x|^{\alpha}}{2^{\alpha}\Gamma(\alpha+1)}$,
valid for all $x\in\R$,   we deduce that
$$
|j_k(x)|\leq\sqrt{\pi}\frac{|x|^k}{2^{k+1}\Gamma(k+3/2)}, \ k\in\N.
$$
Using the bound $\Gamma(x)\geq\sqrt{2\pi}x^{x-1/2}e^{-x}$ we get
\begin{equation}
\label{eq:boundjn}
|j_k(x)|\leq \frac{e^{k+3/2}}{\sqrt{2}(2k+3)^{k+1}}|x|^k, \ k\in\N. 
\end{equation}

We have the following lemma.

\begin{lemma}
\label{lem:legendre}
Write $\widehat{\psi_{n,c}}=\sum_{k\geq 0}\beta_k^n(c)P_{k,c}$. Then for every $k,\ell\geq 0$
$$
|\beta_k^n|\leq \frac{10}{c^{3/2}|\lambda_n(c)|}\left(\frac{e}{2k+3}\right)^{k+1}
$$
\end{lemma}

This bound is an adaptation of techniques from \cite{JKS} to improve the proof of the exponential decay from 
\cite{Xiao}.

\begin{proof}
Using \eqref{eq:prolateFourier}, we have
\begin{eqnarray*}
\beta_k^n(c)&=&\scal{\psi_{n,c},P_{k,c}}_{L^2(I)}
=\int_{-c}^{c} \widehat{\psi_{n,c}}(x)P_{k,c}(x)\,\mbox{d}x\\
&=&(-1)^k\frac{\mu_n(c)}{|\mu_n(c)|^2}\frac{2\pi}{c}\sqrt{\frac{2k+1}{2c}}\int_{-c}^{c} \psi_{n,c}(x/c)P_k\left(\frac{x}{c}\right)\,\mbox{d}x\\
&=&(-1)^k\pi\frac{\mu_n(c)}{c^{1/2}|\mu_n(c)|^2}\sqrt{4k+2}\int_{-1}^{1} \psi_{n,c}(x)P_k(x)\,\mbox{d}x\\
&=&\frac{(-1)^k\pi\sqrt{4k+2}}{c^{1/2}|\mu_n(c)|^2}\int_{-1}^{1}
\int_{-1}^1\psi_{n,c}(y)e^{-icxy}\,\mbox{d}y\,P_k(x)\,\mbox{d}x
\end{eqnarray*}
with \eqref{eq:eigen}. Recalling that $\lambda_n(c)=\dst\frac{c}{2\pi}|\mu_n(c)|^2$ and using Fubini, we get
\begin{eqnarray*}
\beta_k^n(c)&=&\frac{(-1)^k 2\sqrt{4k+2}}{c^{3/2}\lambda_n(c)}
\int_{-1}^{1}\int_{-1}^1P_k(x)e^{-icxy}\,\mbox{d}x\,\psi_{n,c}(y)\,\mbox{d}y\\
&=& \frac{(-i)^k 4\sqrt{4k+2}}{c^{3/2}\lambda_n(c)}
\int_{-1}^{1} \psi_\ell(y) j_{k}(y)\,\mbox{d}y
\end{eqnarray*}
with \eqref{eq:2.2.9}. But then, from \eqref{eq:boundjn} and Cauchy-Schwarz, we deduce that
\begin{eqnarray*}
|\beta_k^n(c)|
&\leq& \frac{ 4\sqrt{4k+2}}{c^{3/2}\lambda_n(c)}
\left(\int_{-1}^{1} j_{k}(y)^2\,\mbox{d}y\right)^{1/2}\\
&\leq& \frac{ 4\sqrt{2k+1}e^{k+3/2}}{(2k+3)^{k+1}c^{3/2}\lambda_n(c)}
\left(\int_{-1}^{1}|y|^{2k}\,\mbox{d}y\right)^{1/2}\\
&=&4\sqrt{2e}\frac{1}{c^{3/2}\lambda_n(c)}\left(\frac{e}{2k+3}\right)^{k+1}.
\end{eqnarray*}
As $4\sqrt{2e}\leq 10$, the result follows.
\end{proof}

We will also need the following estimate.

\begin{lemma}\label{lemma_eigs}
The eigenvalues \eqref{eq:eigen} of $\mathcal{Q}_c$ satisfy
\begin{equation}\label{eq_lemma_eigs}
\Lambda_N:=\left(
\sum_{n=0}^N \frac{1}{\lambda_n(c)}
\right)^{1/2}
\leq \begin{cases}
\sqrt{3+ec} &\mbox{if } N \leq \max(ec,2)\\
\left(\frac{2(N+1)}{ec}\right)^{\frac{2N+1}{2}}&\mbox{if }N\geq \max(ec,2)
\end{cases}.
\end{equation}
\end{lemma}

\begin{proof}
Precise pointwise estimates of the $\lambda_n(c)$'s have been obtained in \cite[Section 4 \& Appendix C]{JKS}
and have been further improved in \cite{BJK} to
$$
\lambda_n(c)\leq \left(\frac{ec}{2(n+1)}\right)^{2n+1}\qquad\mbox{for }n\geq\max\left(n,\frac{ec}{2}\right).
$$
while we always have $\lambda_n(c)<1$.

It follows that
$$
\sum_{k=0}^N\frac{1}{\lambda_n(c)}\geq \begin{cases}N+1\leq 3+ec&\mbox{if }N\leq \max(ec,2)\\
\frac{1}{\lambda_N(c)}\geq \left(\frac{2(N+1)}{ec}\right)^{2N+1}&\mbox{if }N\geq \max(ec,2)
\end{cases}.
$$
The result follows.
\end{proof}

We can now prove our Remez lemma for Prolate spheroidal wave functions.

\begin{theorem}[Remez's Lemma for PSWF]
\label{th_remez_prolates}
Let $N$ be an integer and 
$$
V_N=\mbox{span}\{\psi_{0,c},\ldots,\psi_{N,c}\}\subset PW_c.
$$
Then, for every $\psi\in\widehat V_N$ and every $E\subseteq [-c,c]$ of positive measure,
\begin{align}\label{eq_remez}
\norm{\psi}\leq 2 \left(\frac{8c}{|E|}\right)^{K(N)}\norm{\widehat{\psi}\un_E},
\end{align}
where 
\begin{equation}\label{Theta}
K(N) = \begin{cases}
\max\left(\left\lceil\frac{3200(3+ec)}{c^3}\right\rceil,\left\lceil\frac{4ec}{|E|}\right\rceil\right)
 & \mbox{if }N \leq \max(2,ec),\\
\max\left(20,N,\left\lceil\frac{8(N+1)}{|E|}\right\rceil\right)
& \mbox{if } N \geq \max(2,ec).
\end{cases}
\end{equation}
\end{theorem}

\begin{proof}
Let $\psi=\sum_{n=0}^Nc_\ell\psi_{n,c}$ so that, by orthogonality and the fact that $\norm{\psi_{n,c}}=\lambda_n(c)^{-1/2}$,
$$
\norm{\psi}=\left(\sum_{n=0}^N\frac{|c_n|^2}{\lambda_n(c)}\right)^{1/2}.
$$
On the other hand
$$
\widehat{\psi}=\sum_{n=0}^Nc_\ell\widehat{\psi_{n,c}}
=\sum_{n=0}^Nc_\ell\sum_{k\geq 0}\beta_k^n(c) P_{k,c}.
$$
Let $K$ be an integer that will be fixed later and write
$$
\widehat{\psi}=\sum_{n=0}^Nc_\ell\sum_{k=0}^K\beta_k^n(c)P_{k,c}
+\sum_{n=0}^Nc_\ell\sum_{k>K}\beta_k^n(c)P_{k,c}
:=F_K+R_K.
$$
Note that $F_K$ is a polynomial of degree $K$ so that
\begin{equation}
\label{eq:rem1}
\norm{F_K\un_{[-c,c]}}%
\leq \left(\frac{8c}{|E|}\right)^{K+\frac12}\norm{F_K\un_E} %
\end{equation}
by \eqref{eq:remez}. On the other hand,
$$
R_K=\sum_{k>K}\left(\sum_{n=0}^Nc_n\beta_k^n(c)\right)  P_{k,c}
$$
so that
$$
\norm{R_K\un_E}\leq\norm{R_K\un_{[-c,c]}}
=\left(\sum_{k>K}\abs{\sum_{n=0}^Nc_n\beta_k^n(c)}^2\right)^{1/2}
\leq\left(\sum_{k>K}\sum_{n=0}^N\lambda_n(c)\abs{\beta_k^n(c)}^2\right)^{1/2}\norm{\psi}
$$
by Cauchy-Schwarz. We now apply Lemma \ref{lem:legendre} to get
\begin{eqnarray*}
\norm{R_K\un_E}&\leq& 
\frac{10}{c^{3/2}}\left(\sum_{k>K}\sum_{n=0}^N\frac{1}{\lambda_n(c)}
\left(\frac{e}{2k+3}\right)^{2k+2}\right)^{1/2}\norm{\psi}\\
&=&\frac{10}{c^{3/2}} \left(\sum_{n=0}^N\frac{1}{\lambda_n(c)}\right)^{1/2}
\left(\sum_{k>K}\left(\frac{e}{2k+3}\right)^{2k+2}\right)^{1/2}\norm{\psi}\\
&\leq&\frac{12}{c^{3/2}} \left(\sum_{n=0}^N\frac{1}{\lambda_n(c)}\right)^{1/2}
\left(\frac{e}{2K+5}\right)^{K+1}\norm{\psi}.
\end{eqnarray*}
Using Lemmas \ref{lemma_eigs} and \ref{lem:legendre} we can rewrite this in the form $\norm{R_K\un_E}\leq\Lambda_N\Phi_K\norm{\psi}$ 
with
$$
\Lambda_N:=\begin{cases}
\sqrt{3+ec} &\mbox{if } N \leq \max(ec,2)\\
\left(\frac{2(N+1)}{ec}\right)^{N+\frac{1}{2}}&\mbox{if }N\geq \max(ec,2)
\end{cases}
\quad\mbox{and}\quad \Phi_K=\frac{12}{c^{3/2}}\left(\frac{e}{2K+5}\right)^{K+1}.
$$
Next
\begin{eqnarray*}
\norm{\widehat{\psi}\un_E}&\geq&\norm{F_K\un_E}-\norm{R_K\un_E}
\geq\left(\frac{|E|}{8c}\right)^{K+\frac{1}{2}}\norm{F_K\un_{[-c,c]}} -\norm{R_K\un_{[-c,c]}}\\
&\geq&\left(\frac{|E|}{8c}\right)^{K+\frac{1}{2}}\norm{\widehat{\psi}}-
\left(1+\left(\frac{|E|}{8c}\right)^{K+\frac{1}{2}}\right)\norm{R_K\un_{[-c,c]}}\\
&\geq&\left(\frac{|E|}{8c}\right)^{K+\frac{1}{2}}\norm{\widehat{\psi}}-2\norm{R_K\un_{[-c,c]}}
\end{eqnarray*}
since $E\subset[-c,c]$ implies $\left(\frac{|E|}{8c}\right)^{K+\frac{1}{2}}\leq1$.
Therefore
$$
\norm{\widehat{\psi}\un_E}
\geq\left(\frac{|E|}{8c}\right)^{K+\frac{1}{2}}\left(1-2\Lambda_N\Phi_K
\left(\frac{8c}{|E|}\right)^{K+\frac{1}{2}}\right)\norm{\widehat{\psi}}.
$$
It remains to choose $K$ so that $\dst\Lambda_N\Phi_K\leq\frac{1}{4}\left(\frac{|E|}{8c}\right)^{K+\frac{1}{2}}$.

First, if $N\leq \max(ec,2)$, then we want
$$
\left(\frac{e}{2K+5}\right)^{1/2}\left(\frac{e}{2K+5}\right)^{K+1/2}=\left(\frac{e}{2K+5}\right)^{K+1}\leq \frac{c^{3/2}}{48\sqrt{3+ec}}
\left(\frac{|E|}{8c}\right)^{K+\frac{1}{2}}
$$
so that it is enough that
$\dst\frac{e}{2K+5}\leq \frac{c^3}{48^2(3+ec)}$ and
$\dst\frac{e}{2K+5}\leq \frac{|E|}{8c}$ so we take
$$
K=K(N):=\max\left(\left\lceil\frac{3200(3+ec)}{c^3}\right\rceil,\left\lceil\frac{4ec}{|E|}\right\rceil\right).
$$

On the other hand, if $N\geq \max(ec,2)$, then we want
$$
\left(\frac{e}{2K+5}\right)^{1/2}\left(\frac{e}{2K+5}\right)^{K+1/2}=\left(\frac{e}{2K+5}\right)^{K+1}\leq 
\frac{1}{4}\left(\frac{ec}{2(N+1)}\right)^{N+\frac{1}{2}}\left(\frac{|E|}{8c}\right)^{K+\frac{1}{2}}
$$
Taking $K:=K(N)= \max\left(20,N,\left\lceil\frac{8(N+1)}{|E|}\right\rceil\right)$, we get
$$
\left(\frac{e}{2K+5}\right)^{K+1}\leq
\frac{1}{4}\left(\frac{e}{2K}\right)^{K+1/2}
\leq \frac{1}{4}\left(\frac{ec}{2(N+1)}\frac{|E|}{8c}\right)^{K+1/2}
$$
which gives the desired estimate since $2(N+1)>ec$ and $K\geq N$.
\end{proof}

\subsection{Sampling the heat flow}

Equipped with the Remez-Tur\'an Property, we are ready to close the blind spots in Theorem \ref{sufconT}. We do it only in the case of heat flow as it should be clear how to obtain similar estimates in the case of other kernels $\phi\in\Phi$.

\begin{theorem}
\label{thm:gauss_body}
Let $\widehat\phi(\xi)=e^{-\sigma^2\xi^2}$, $\sigma\not=0$, and $m \geq 2$ be an integer. 
Let $V=V_N$ be given by \eqref{eq_moda}, \eqref{eq_modb}, or \eqref{eq_modc}. 
Then, for every $f\in V$,
\begin{equation}
\kappa\norm{\widehat f}^2 \le 
\int_0^1 \sum_{k\in\ZZ}\left|f_t\left(\frac{m\pi}{c}k\right)\right|^2\,\mathrm{d}t \le \|\widehat f\|^2,
\label{eq:heatsubspace}
\end{equation}
where
\begin{align}
\label{eq_AAAAA}
\kappa =
\frac{c\kappa_0(c)}{(\sigma c)^2+m}
\exp\Bigl(-\kappa_1(c)N-m^2\bigl(-\kappa_2(c)\ln\sigma+\kappa_3(c)\sigma^2+\ln m\bigr)\Bigr)
\end{align}
with $\kappa_j$ positive constants that depend on $c$ only.
\end{theorem}

\begin{remark}
For $V=V_N$ given by \eqref{eq_modb}, \eqref{eq_modc} and for $V=V_N$ given by \eqref{eq_moda}
when $N\geq \max(2,ec)$, $\kappa_0,\kappa_1$ do not depend on $c$.
\end{remark}

\begin{proof}
To obtain this result, we take $\eta =\dst \frac{1}{8}$ in Proposition \ref{prop:gauss}.
First note that if 
$\widetilde{E} = \dst\left[-\frac{3}{8} ,-\frac{1}{8}\right]\cup\left[\frac{1}{8}, \frac{3}{8}\right]$ and 
$E=\dst\left(\frac{2c}{m}(\tilde E+\Z)\right)\cap[-c,c]$ then $\dst\frac{c}{|E|}\geq \frac{1}{8}$
(say). Then \eqref{eq_ppp} tells us that
$$
A\norm{\widehat f\un_{E}}^2 \le \int_0^1 \sum_{k\in\ZZ}\left|f_t\left(\frac{m\pi}{c}k\right)\right|^2\,\mathrm{d}t \le \|\widehat f\|^2,\leqno\eqref{eq_ppp}
$$
for any $f\in PW_c$, where
\begin{equation*}
A= \frac{c}{2e\pi^2(2(\sigma c)^2+m)}\frac{(cR/2)^{m(m-1)}}{{m^{1-m+2m^2}}}
\quad\mbox{with}\quad
R = 2\sigma^2\min\left\{\frac{1}{8}e^{-(\sigma/8)^2}, ce^{-(\sigma c)^2}\right\}.
\end{equation*}
Note that
$$
\frac{cR}{2}=\min\left\{c\frac{\sigma^2}{8}e^{-(\sigma/8)^2}, c^2\sigma^2e^{-(\sigma c)^2}\right\}<1
$$
so that 
$$
\frac{(cR/2)^{m(m-1)}}{{m^{1-m+2m^2}}}\geq \left(\frac{cR/2}{m}\right)^{m^2}=\exp\Bigl(-m^2\bigl(-\gamma_1(c)\ln\sigma+\gamma_2(c)\sigma^2+\ln m\bigr)\Bigr).
$$
Finally $A\geq  \frac{c\gamma_0}{\bigl((\sigma c)^2+m\bigr)}\exp\Bigl(-m^2\bigl(-\gamma_1(c)\ln\sigma+\gamma_2(c)\sigma^2+\ln m\bigr)\Bigr)$ where $\gamma_0,\gamma_1(c),\gamma_2(c)$ are constants depending on $c$ only.

It remains to fix the blind spots $\norm{\widehat f\un_{E}}^2$ with the help of a Remez type inequality.
For $V=V_N$ given by \eqref{eq_modb}, \eqref{eq_modc} and $f\in V_N$, we simply have
$\norm{\widehat f\un_{E}}^2\geq \gamma_3^{2N+1}\norm{f}^2$ where $\gamma_3<1$ is a constant.

For $V=V_N$ given by \eqref{eq_moda}, $\norm{\widehat f\un_{E}}^2\geq \gamma_3^{2K(N)}\norm{f}^2$
where $K(N)$ is given by \eqref{Theta}
$$
K(N) = \begin{cases}
\max\left(\left\lceil\frac{3200(3+ec)}{c^3}\right\rceil,\left\lceil\frac{4ec}{|E|}\right\rceil\right)\leq\gamma_4(c)
 & \mbox{if }N \leq \max(2,ec),\\
\max\left(20,N,\left\lceil\frac{8(N+1)}{|E|}\right\rceil\right)\leq 64(N+1)
& \mbox{if } N \geq \max(2,ec).
\end{cases}
$$
Adding the estimates for fixing the blind spot yields \eqref{eq_AAAAA}.
\end{proof}

\begin{remark}
Theorem \ref{thm:gauss_body} immediately implies Theorem \ref{thm:gauss}. We also note that if $V=V_N$ is given by \eqref{eq_moda}  or \eqref{eq_modc}, the reconstruction can be done from measurements at a finite number of spacial locations. Indeed, our results imply that in this case one can find the coefficients of $f$ in its decomposition in a basis of $V$ via simple least squares. 
\end{remark}

\section{Sensor density, maximal spatial gaps and condition numbers}\label{gap}

In this section, we discuss irregular spatio-temporal sampling. We establish that stable reconstruction from dynamical samples may occur when the set $\Lambda$ has an arbitrarily small density. More importantly, however, we show that the density cannot be arbitrarily small for fixed frame bounds in \eqref{SemiContFrL}. In fact, we provide an explicit estimate for the maximal spatial gap in terms of the condition number $\frac BA$.

\begin {example} 
\label{ex_low_dens}
In this example, we take $c=1/2$ to simplify discussion.
Assume that $\phi\in\Phi$ is such that $\widehat{\Phi}$ is real, even, and decreasing on $[0,1/2]$. Let $\Lambda_0=m \Z$,  with $m \in \mathbb{N}$ odd,
$\Lambda_k=mn\Z+k$,  where $n$ is any fixed odd number and $k=1,\dots\frac {m-1} 2$ . Then  $\Lambda=\bigcup\limits_{k=0}^{\frac {m-1} 2} \Lambda_k$ has density $D^{-}(\Lambda) \leq 1/n+1/m$ %
and is a stable set of sampling, i.e.,  \eqref {SemiContFr}  is satisfied.
\end {example}

The claim in the last example follows by stringing together several theorems on dynamical sampling.  Firstly, \cite [Theorems 2.4 and 2.5]{ADK15} yield that any $f \in \ell^2(\Z)$ can be recovered from the space-time samples $\{\phi^j\ast f(x_k): j=0,\dots,m-1, \; x_k \in \Lambda \}$ and that  the problem of sampling and reconstruction in $PW_c$  on subsets of $\Z$ is equivalent to the sampling and reconstruction problem of sequences in $\ell^2(\Z)$. Secondly, combining  \cite  [Theorems 5.4 and 5.5] {AHP19}  shows that for $\phi\in\Phi$, $f\in PW_c$ can be stably reconstructed from $\{\phi^j\ast f(x_k): j=0,\dots,m-1, \; x_k \in \Lambda \}$ if and only if \eqref {SemiContFr}  is satisfied.

Example \ref{ex_low_dens} thus shows that \eqref{SemiContFr} can hold with sets having arbitrarily small densities. The goal of this section is to show that the maximal gap in such sets is controlled by the condition number $B/A$.

We first establish the following lemma, which parallels \cite[Proposition 4.4]{GRUV15}.

\begin{lemma}
\label{lem:estsinc}
Let $\phi\in \Phi$ be such that $\widehat\phi$ is $\mathcal{C}^1$-smooth on $I = [-c,c]$. 
Then there exists a finite constant $C_{\phi, L}$ such that 
	\begin{equation}
	\label{upest}
	\int_{0}^{L}|(\sinc (c\cdot)*\phi_t)(x)|^2\,\mathrm{d}t\leq \frac{C_{\phi, L}}{1+x^2}, \text{ for all }x\in\mathbb{R}.
	\end{equation}
On the other hand, setting $c_{\phi,L}=  \frac{2(\newmphi^{2L}-1)}{\pi^2\ln\newmphi}>0$, for
$|x|\leq \pi/2c$, we have
\begin{equation}\label{loest}
\int_{0}^{L}|(\sinc*\phi_t)(x)|^2\,\mathrm{d}t\geq c_{\phi,L}.
\end{equation}
\end{lemma}

\begin{proof}
Firstly, writing the Fourier inversion formula shows that
\begin{equation}
(\sinc(c\cdot)*\phi_t)(x)=\frac{1}{2c}\int_{-c}^{c}\bigl(\widehat \phi(\xi)\bigr)^te^{ix\xi}\,\mbox{d}\xi
\label{eq:sincconv}
\end{equation}
from which it follows that 
\begin{equation}\label{firstest}
|(\sinc(c\cdot)*\phi_t)(x)|\leq \frac{1}{2c}\int_{-c}^{c}|\widehat \phi(\xi)|^t\,\mbox{d}\xi\leq 1,
\end{equation}
due to $|\widehat\phi|\leq 1$.

Secondly, note that, due to its smoothness, $\widehat\phi^\prime$ is bounded 
by $E_\Phi:=\sup_{\xi\in[-c,c]}|\widehat\phi^\prime(\xi)|<+\infty$ on $[-c,c]$
Then, integrating \eqref{eq:sincconv} by parts leads to
\[
 x(\sinc(c\cdot)*\phi_t)(x) %
=\frac{\widehat{\phi}^t(c)e^{i cx}-\widehat{\phi}^t(-c)e^{-i cx}}{2 ic} 
-\frac{1}{2ic}\int_{-c}^{c}e^{i x\xi}t\widehat{\phi}^{t-1}(\xi)\widehat{\phi}^\prime(\xi)\,\mbox{d}\xi,
\]
and, as $\newmphi\leq\widehat \phi\leq 1$ on $I$, we deduce that
\[
|x(\sinc(c\cdot)*\phi_t)(x)|\leq \frac{1}{c}+\frac{E_\phi}{\newmphi}t.
\]
Consequently,
\[
x^2\int_0^L|(\sinc(c\cdot)*\phi_t)(x)|^2\,\mbox{d}t\leq \int_0^L\left(\frac{1}{c}+\frac{E_\phi}{\newmphi}t\right)^2dt
=\frac{\newmphi}{3E_\phi}\left(\left(\frac{1}{c}+\frac{E_\phi}{\newmphi}L\right)^3- \frac1{c^3}\right),
\]
and the estimate \eqref{upest} follows in view of \eqref{firstest}.

On the other hand \eqref{eq:sincconv}  implies that
$$
|(\sinc(c\cdot)*\phi_t)(x)| \geq |\Re (\sinc(c\cdot)*\phi_t)(x)|
= \left|\frac{1}{2c}\int_{-c}^{c}\widehat \phi(\xi)^t\cos x\xi \,\mbox{d}\xi\right|.
$$
But, for $|\xi|\leq c$, we have $\widehat\phi(\xi)^t\geq \newmphi^t$. Further, if we also have
$|x|\leq \pi/2c$, then $\cos 2x\xi\geq 0$. Therefore,
\[
|(\sinc(c\cdot)*\phi_t)(x)| \geq
\frac{1}{2c}\int_{-c}^{c}\widehat \phi(\xi)^t\cos x\xi \,\mbox{d}\xi
\geq \newmphi^t\frac{1}{2c}\int_{-c}^{c}\cos x\xi \,\mbox{d}\xi
=\newmphi^t\sinc(cx)\geq \frac{2}{\pi}\newmphi^t
\]
since $\sinc(cx)$ is decreasing on $[0,\pi/2c]$ and $\dst\sinc\left(c\frac{\pi}{2c}\right)=\frac{2}{\pi}$.
It follows that
\[
\int_0^L|(\sinc*\phi_t)(x)|^2\,\mbox{d}t\geq \frac{4}{\pi^2}\int_0^L \newmphi^{2t}dt
=  \frac{2(\newmphi^{2L}-1)}{\pi^2\ln\newmphi}> 0,
\]
and %
 we get the desired result.
\end{proof}

\begin{remark}
If $\widehat\phi(\xi)=e^{-\sigma^2\xi^2}$, $\sigma\not=0$, then $\newmphi = e^{-(\sigma c)^2}$ 
and we may take $E_\phi = \sqrt{\frac{2}{e}}|\sigma|$. Therefore, the constants $c_{\phi,L}$ and $C_{\phi,L}$
in the above lemma can be taken as
\begin{equation}\label{cest}
C_{\phi,L} = \frac{1}{c^3}+\bigl(1+\sigma^2e^{(\sigma c)^2}\bigr) L^3
\quad \mbox{ and } \quad 
c_{\phi,L} = \frac{2(1-e^{-2L(\sigma c)^2})}{\pi^2(\sigma c)^2} .
\end{equation}
For the estimate of $C_{\phi,L}$ we have used that
$$
\frac{1}{3\alpha}(a+\alpha b)=a^2b+\alpha ab^2+\frac{\alpha^2}{3}b^3
\leq a^3+\frac{b^3}{3}(1+2\alpha^{3/2}+\alpha^2)\leq a^3+b^2(1+\alpha^2)
$$
with H\"older.
\end{remark}

\begin{theorem}\label{ConGapThm}
Let $\phi\in\Phi$ and assume that $\widehat\phi$ is $\mathcal{C}^1$-smooth on $[-c,c]$. 
Assume that $\Lambda\subseteq\mathbb{R}$ is a stable sampling set for  Problem \ref{pro1} with frame bounds 
$A$, $B$ (i.e., \eqref{SemiContFrL} holds:
\begin{equation*}
A \norm{f}_2^2 \leq \int_0^L \sum_{\lambda \in \Lambda} \abs{(f*\phi_t)(\lambda)}^2 dt \leq B \norm{f}_2^2, 
\text{ for all } f \in PW_c.
\leqno\eqref{SemiContFrL}
\end{equation*}
Let $c_{\phi,L}$ and $C_{\phi,L}$ be the constants from Lemma \ref{lem:estsinc}. Then for
$R\geq \dst\max\left(\frac{\pi}{c},\frac{8c}{\pi}\frac{B}{A}\frac{C_{\Phi,L}}{c_{\Phi,L}}\right)$ and every 
$a\in\mathbb{R}$, we have $[a-R,a+R]\cap\Lambda\neq\emptyset$.
Further, we have $D^-(\Lambda)\geq \dst\min\left(\frac{c}{2\pi},
\frac{\pi}{16c}\frac{A}{B}\frac{c_{\Phi,L}}{C_{\Phi,L}}\right)$ and $\dst D^+(\Lambda)\leq 4\frac{B}{c_{\Phi,L}}$.
\end{theorem}

\begin{proof}
Denoting $I_a = [a-\pi/4c,a+\pi/4c]$, $a\in\R$, let us bound the covering number
\[
n_{\Lambda}:=\sup\limits_{a\in\mathbb{R}}\#\left(\Lambda\cap I_a \right). 
\] 

We use \eqref{loest}, i.e.,~the fact that $\dst\int_0^L|(\sinc(c\cdot)*\phi_t)(x)|^2\,\mbox{d}t\geq c_{\Phi,L}$ 
for $|x|\leq \pi/2c$, and our first observation to obtain 
\begin{eqnarray*}
\#\left(\Lambda\cap I_a \right)&\leq &\frac{1}{c_{\Phi,L}}
\sum\limits_{\lambda\in\Lambda\cap I_a}\int_0^L|(\sinc (c\cdot)*\phi_t)(\lambda-a)|^2\mbox{d}t\\
&\leq& 
\frac{1}{c_{\Phi,L}}
\sum\limits_{\lambda\in\Z}\int_0^L|(\sinc (c\cdot)*\phi_t)(\lambda-a)|^2\mbox{d}t
\leq \frac{B}{c_{\Phi,L}}\norm{\sinc c(t - a)}^2
\end{eqnarray*}
where we applied \eqref{SemiContFrL} to $f(t)=\sinc c(t - a)$ for all $a\in\mathbb{R}$.
As $\hat f(\xi)=\dst\frac{\pi}{c}\un_{[-c,c]}$, Parseval's relation gives
$\norm{f}^2%
=\frac{\pi}{c}$ hence
\begin{equation}\label{upp_Bnd_Fn1}
n_{\Lambda}\leq\frac{\pi}{c}\frac{ B}{c_{\Phi,L}}.
\end{equation}
As a first consequence, this implies that $D^+(\Lambda)\leq 4\frac{B}{c_{\Phi,L}}$.

Now we assume that for some $a_0\in\mathbb{R}$, and some $R\ge\dst\frac{\pi}{c}$, $\Lambda\cap[a_0-R,a_0+R]=\emptyset$.
As the Paley-Wiener space is invariant under translation, if \eqref{SemiContFrL} holds for $\Lambda$,
it also holds for its translates, so that we may assume that $a_0=0$.

From Lemma \ref{lem:estsinc}, there exists $C_{\Phi,L}$ such that 
$\dst\int_{0}^{L}|(\sinc (c\cdot)*\phi_t)(x)|^2\mathrm{d}t\leq C_{\Phi,L}/(1+x^2)$. Therefore, we have 
the following estimates
\begin{eqnarray*}
\frac{\pi}{c}A&\leq&\sum\limits_{\lambda\in\Lambda}\int_0^L|(\sinc(c\cdot)*\phi_t)(\lambda)|^2\,\mbox{d}t 
\leq\sum\limits_{\lambda\in\Lambda}\frac{C_{\Phi,L}}{1+\lambda^2}\\
&\leq&\sum\limits_{k=0}^\infty\sum\limits_{\lambda \in\Lambda\cap [R+k\pi/2c,R+(k+1)\pi/2c]}
\frac{C_{\Phi,L}}{1+\lambda^2}
+\sum\limits_{k=0}^\infty\sum\limits_{\lambda\in\Lambda\cap [-R-(k+1)\pi/2c,-R-k\pi/2c]}
\frac{C_{\Phi,L}}{1+\lambda^2}\\
&\leq&2n_{\Lambda}\sum\limits_{k=0}^\infty\frac{C_{\Phi,L}}{1+(R+k\pi /2c)^2}
\leq  4\frac{C_{\Phi,L}B}{c_{\Phi,L}}\int_{R-\pi/2c}^{\infty}\frac{\mathrm{d}x}{1+x^2}\\
&\leq& 4\frac{C_{\Phi,L}B}{c_{\Phi,L}}\int_{R/2}^{\infty}\frac{\mathrm{d}x}{x^2}=8\frac{C_{\Phi,L}B}{c_{\Phi,L}R}
\end{eqnarray*}
since we assumed that $R\geq\pi/c$. It follows that
$R\leq\dst
\frac{8c}{\pi}\frac{B}{A}\frac{C_{\Phi,L}}{c_{\Phi,L}}$.
Finally, note that this implies that $D^-(\Lambda)\geq \frac{1}{2R}$.
\end{proof}

\begin{remark}
Computing the explicit estimate for $\frac{C_{\Phi,L}}{c_{\Phi,L}}$, we observe that the maximal allowed gap in spacial measurements grows  with $L$, which is to be expected. For the Gaussian, we may take the constant 
$\frac{C_{\Phi,L}}{c_{\Phi,L}}$ to be $O(L^2)$ (see \eqref{cest}).
The above results also shows that for $\mathcal{C}^1$-smooth functions $\phi$, stable sampling sets must have positive lower density. %
\end{remark}
\begin{remark}
Theorem \ref{ConGapThm} immediately implies Theorem \ref{thm_gap_intro}.
\end{remark}

\section{Acknowledgments.}
K.\ G.\ was
  supported in part by the  project P31887-N32  of the
Austrian Science Fund (FWF), and J.~L.~R.~gratefully acknowledges
support from the Austrian Science Fund (FWF): Y 1199 and P 29462. 

The authors are also grateful for the hospitality of various conferences, where we were able to work together on this project. We thank the hosts and organizers of all those events, in particular: ICERM, University of Bordeaux,  and Vanderbilt University.

Finally, it gives us great pleasure to dedicate this paper to  Guido Weiss on the occasion of his $90^{th}$ birthday. To a dear friend who taught so much to so many: Merry Guidmas!

\bibliographystyle{plain}

\end{document}